\documentclass[10pt,letterpaper]{article}
\usepackage[utf8]{inputenc}
\usepackage{amsmath}
\usepackage{amsfonts}
\usepackage{amsthm,amssymb}
\usepackage{enumerate}
\usepackage{graphicx}
\DeclareGraphicsExtensions{.png,.eps,.pdf}
\graphicspath{{imagenespdf/}}
\usepackage{subcaption}
\usepackage{microtype}
\usepackage{wrapfig}
\usepackage[symbol]{footmisc}
\usepackage[top=1.8cm, bottom=1.8cm, right=2cm, left=2cm]{geometry}

\newtheorem{teo}{Theorem}[section]
\newtheorem{lem}{Lemma}[section]
\newtheorem{cor}{Corollary}[section]
\newtheorem{defi}{Definition}[section]

\newtheorem{prop}{Proposition}[section]
\newtheorem{obs}{Observation}[section]
\newtheorem{notac}{Notation}[section]

\newtheorem{remark}{Remark}[section]

\begin{document}

\begin{center}
	\LARGE{Cycles of length 3 and 4 in edge-colored complete graphs with restrictions in the color transitions\footnote[1]{Research supported by grants CONACYT FORDECYT-PRONACES/39570/2020 and UNAM-DGAPA-PAPIIT IN102320 (H. Galeana-Sánchez), and CONACYT scholarship for postgraduate studies 782604 (F. Hernández-Lorenzana)}}
\end{center}

\begin{center}
	$\text{Hortensia Galeana-Sánchez}^{\text{a}}, \text{Felipe Hernández-Lorenzana}^{\text{a,}}\footnotemark[2], \text{Rocío Sánchez-López}^{\text{b}}$
\end{center}
\vspace*{0.2cm}
\begin{center}
	\footnotesize{$^{a}$Instituto de Matemáticas, Universidad Nacional Autónoma de México, Área de la investigación científica, Circuito Exterior, Ciudad Universitaria, 04510 Coyoacán, CDMX, Mexico\\
	$^{b}$Facultad de Ciencias, Universidad Nacional Autónoma de México, Ciudad Universitaria, Circuito Exterior s/n, 04510 Coyoacán, CDMX, Mexico }\footnotetext[2]{Corresponding author} \footnote[0]{Email adresses: \texttt{hgaleana@matem.unam.mx} (H. Galeana-Sánchez), \texttt{felipehl@ciencias.unam.mx} (F. Hernández-Lorenzana), \texttt{usagitsukinomx@yahoo.com.mx} (R. Sánchez-López)}
\end{center}

\begin{center}
	\rule{\textwidth}{0.5mm}
\end{center}

\begin{abstract}
Let $G$ be an edge-colored graph, a walk in $G$ is said to be a properly colored walk iff each pair of consecutive edges have different colors, including the first and the last edges in case that the walk be closed. Let $H$ be a graph possible with loops. We will say that a graph $G$ is an $H$-colored graph iff there exists a function $c:E(G)\longrightarrow V(H)$. A path $(v_1,\cdots,v_k)$ in $G$ is an $H$-path whenever $(c(v_1v_2),\cdots,$ $c(v_{k-1}v_k))$ is a walk in $H$, in particular, a cycle $(v_1,\cdots,v_k,v_1)$ is an $H$-cycle iff  $(c(v_1 v_2),\cdots,c(v_{k-1}v_k),$ $c(v_kv_1), c(v_1 v_2))$ is a walk in $H$. Hence, $H$ decide which color transitions are allowed in a walk, in order to be an $H$-walk. Whenever $H$ is a complete graph without loops, an $H$-walk is a properly colored walk, so $H$-walk is a more general concept. In this paper, we work with $H$-colored complete graphs, with restrictions given by an auxiliary graph. The main theorems give conditions implying that every vertex in an $H$-colored complete graph, is contained in an $H$-cycle of length 3 and in an $H$-cycle of length 4. As a consequence of the main results, we obtain some well-known theorems in the theory of properly colored walks.
\end{abstract}

\textbf{Keywords:} Edge-colored graph, $H$-cycle, Properly colored walk

\section{Introduction}

We assume that the reader is familiar with the standard terminology on graph theory, and for further notation and terminology not defined here, we refer the reader to~\cite{BJM}. In this paper we work with simple finite graphs. Whenever $G$ is a graph, $V(G)$ will denote the set of vertices and $E(G)$ the set of edges of $G$. The order of the graph $G$ is its number of vertices $|V(G)|$.

A graph $G$ of order $n$ is said to be pancyclic iff it contains a cycle of length $l$ for every $l$ in $\{3,\cdots,n\}$, and $G$ is called vertex-pancyclic whenever each vertex of $G$ belongs to a cycle of length $l$ for each $l$ in $\{3,\cdots,n\}$. The pancyclicity and the vertex-pancyclicity in graphs have been deeply studied since 1971, when Bondy conjectured in \cite{BJA} that almost every non-trivial condition implying that a graph $G$ is hamiltonian, also implies that $G$ is pancyclic. At least 127 papers have been published in these topics, see for example \cite{BFG}, \cite{BJA1} and \cite{BJA2}. Bondy's metaconjecture has been verified by adding some extra conditions, and since vertex-pancyclicity implies pancyclicity and pancyclicity implies hamiltonicity, this motivates the research of conditions implying the vertex-pancyclicity in a graph. In \cite{LCM}, Ming-Chu Li et al. proved that determining wheter a graph is pancyclic (or vertex-pancyclic) is an $NP$-complete problem, even for 3-connected cubic planar graph. 

We can think in the following application: a travel agency has offices in $n$ cities, between some of which it is possible to travel directly (avoiding any other city). To offer a large variety of trips, the tourism agency want to determine if for every $k$ in $\{3,\cdots,n\}$ there exists a tour starting and ending in the same city, and passing by each of $k$ cities exactly once in some order. Moreover, they want to determine if this is possible to do it regardless the departure city. Therefore, in terms of Graph Theory, we want to determine whether the graph associated with the problem is pancyclic or vertex-pancyclic.

Now we can consider the following application in relation to the means of transportation: for each two consecutive cities on a given tour, there exists several ways to travel (say bus, plane, train, boat, and so on), and according to the customer's preferences, there will be some transport transitions that should be avoided, and some others that will be needed. For example, if from city $A$ to city $B$ the customer traveled by plane, then from city $B$ to city $C$ the travel must be done by train, or another possibility is that the customer prefers to do the tour exclusively by bus. We can represent this situation with an edge-colored multigraph, where each vertex represents a city, and in order to make more visual the representation of this problem, we will have an edge with color $i$ between two different vertices $A$ and $B$, whenever we can travel from the city $A$ to the city $B$ directly by the $i$-th means of transportation. In this modify problem, it does not suffices to find a cycle of length $k$ for every $k$ in $\{3,\cdots,n\}$, but now we have to find a cycle (tour) of length $k$ according to the restrictions given by the color transitions in the obtained edge-colored multigraph (customer).

In this paper, we work in an edge-colored simple graph $G$ and we deal with the problem of finding conditions implying that, for each vertex $x$ of $G$ there exists an edge-colored cycle of length 3 (respectively 4) containing $x$ with some constraints given in the coloring of the graph.

A walk in an edge-colored graph is a properly colored walk, iff every two consecutive edges have different colors, including the first and last edges when the walk is closed. The theory of properly colored walks has relevance in Graph Theory and Algorithms \cite{HaR} and \cite{WoD}.  The study of applications on properly colored structures starts in 1891, with the remarkable paper by Petersen (see\cite{PeJ}). Properly colored walks have shown to be an effective way to model some real applications in different fields, as we can mention Genetic and Molecular biology \cite{DoD1}, \cite{DDT}, \cite{PPA}, Social Sciences \cite{CMM}, Engineering and Computer Science \cite{ASK}, \cite{SER}, \cite{TCL}, and Management Science \cite{XKH}, \cite{XLH}. There exist some classical results in the theory of properly colored walks, we have for example Kotzig's Theorem \cite{KoA}, which provides a characterization of edge-colored multigraphs, having properly colored closed Euler trails. As closed trails contain cycles, we can ask whether a properly colored trails contain properly colored cycles, and in general, ask for conditions implying the existence of properly colored cycles. Solving this question, in 1983 Grossman and Häggksvit \cite{GJH} proved Theorem~\ref{teo:Yeo} when $c=2$, and later Yeo proved it for every $c\ge 2$. Although the problem of determining the existence of properly hamiltonian cycles, in 2-edge-colored graphs is $NP$-complete (see \cite{BJG}), we can check whether an $c$-edge-colored graph have a properly colored cycle in polynomial time, using a recursive algorithm provided by Yeo's Theorem. Moreover, in \cite{GGK} the authors proved that a shortest properly colored cycle can be found in polynomial time. For a classical survey in this topic, see chapter 16 in \cite{BJG}, and a more recent survey can be found in \cite{BJG1} and \cite{LXM}.

\begin{teo}[\textbf{Yeo}, \cite{YeA}]\label{teo:Yeo} 
	Let $G$ be a $c$-edge-colored graph, $c\ge 2$, with no properly colored cycle. Then, $G$ has a vertex $z$ in $V(G)$, such that no connected component of $G-z$ is joined to $z$, with edges of more than one color. 
\end{teo}

In view of these theorems, it is natural to ask for conditions implying the existence of colored walks, with restrictions in the color transitions. In \cite{SDF}, Szachniuk et al. introduced what they called the Orderly Colored Longest Path problem (OCLE), that is: if $G$ is an edge-colored graph with colors in $C$, $P=(v_1,\cdots,v_n)$ is a path in $G$, and $O=\langle c_1,\cdots,c_k\rangle$ is a fixed sequence of colors for some colors $c_1,\cdots,c_k$ in $C$, then $P$ is an orderly colored path, if the colors of consecutive edges in $P$ follow the color sequence defined by $O$. Notice that, if $P$ is an orderly colored path following the color sequence $O$, and the length of $O$ equals 2 with the two colors of $O$ different, then $P$ is a properly colored path. For other applications motivating the study of colored paths with constraints in the color sequence, see \cite{SDF}.

The following concepts, which are more general than the previous one, were introduced in a work developed by Linek and Sands in \cite{LVS}, in the context of kernels in arc-colored directed graphs.

Let $H$ be a graph possibly with loops and $G$ be a graph. We say that $G$ is an $H$-colored graph whenever there exists a function $c:E(G)\rightarrow V(H)$. A path $(v_1,\cdots,v_n)$ in an $H$-colored graph $G$ is an $H$-path iff $(c(v_1 v_2),c(v_2 v_3),\cdots, $ $c(v_{n-1}v_n))$ is a walk in the graph $H$, in particular, a cycle $(v_1,\cdots,v_n,v_1)$ is an $H$-cycle whenever $(c(v_1 v_2),c(v_2 v_3),\cdots,c(v_{n-1}v_n),c(v_n v_1), c(v_1 v_2))$ is a walk in~$H$. Notice that, if $H$ is the complete graph without loops, then an $H$-path is a properly colored path. Moreover, the concept of $H$-path generalizes the concept of orderly colored path, because if $O=\langle c_1,\cdots,c_k \rangle$ is the sequence of colors given, then it is enough to consider any graph $H$ such that $\{c_1,\cdots, c_k\}\subseteq V(H)$ and $(c_1,\cdots,c_k,c_1)$ is a walk in $H$, to color the edges of $G$.

The study of the existence of certain $H$-walks in $H$-colored graphs, began in \cite{GRS} in a work entitled ``Some Conditions for the Existence of Euler $H$-trails". In \cite{GRS1}, the authors extended Theorem~\ref{teo:Yeo} in the context of $H$-coloring, so they found structural conditions implying the existence of $H$-cycles in an $H$-colored graph, although no information about the length of such a cycles was provided. In~\cite{GHS}, the authors gave conditions implying the existence of an $H$-cycle of length at least $\left\lceil\frac{|V(G)|}{3}\right\rceil+1$ in an $H$-colored graph $G$. An interesting auxiliary graph in working with $H$-colorings was defined by Kotzig in \cite{KoA} and proved be cumbersome, in \cite{GHS}, \cite{GRS} and \cite{GRS1}:

\begin{defi}[\cite{KoA}]		
	Let $H$ be a graph possibly with loops and $G$ be an $H$-colored graph with a fixed $H$-coloring $c:E(G)\longrightarrow V(H)$. For each non-isolated vertex $x$ of $G$, we denote by $G_x$ the graph defined as follows:
	\begin{enumerate}
		\item[(i)] $V(G_x)=\{e\in E(G):e\ \mbox{is incident with}\ x\}$.
		\item[(ii)] For two different vertices $a$ and $b$ in $V(G_x)$, $ab\in E(G_x)$ if and only if $c(a)c(b)\in E(H)$.
	\end{enumerate}
\end{defi}

Notice that for every non-isolated vertex $x$ in $V(G)$, $G_x$ is a simple graph.

Let $H$ be a graph possibly with loops, $G$ an $H$-colored graph, $U=(x_1,\cdots,$ $x_n)$ a walk in $G$ and $i$ in $\{2,\cdots, n-1\}$. We say that $x_i$ is an obstruction of $U$ whenever $c(x_{i-1}x_i)c(x_ix_{i+1})\notin E(H)$. When $W=(x_1,\cdots, x_{n-1},x_1)$ is a closed walk, $x_1$ is an obstruction of $W$ iff $c(x_{n-1}x_1)c(x_1x_2)\notin E(H)$. The interest of defining obstruction of a walk and the auxiliary graph $G_x$, is established in the following observation.

\begin{obs}\label{obs:equiv}
	Let $H$ be a graph possibly with loops, and $G$ be an $H$-colored graph, such that for every $x$ in $V(G)$, $G_x$ is a complete $k_x$-partite graph for some $k_x$ in $\mathbb{N}$. Suppose that $\{ux,vx\}$ is a subset of $E(G)$. The following statements are equivalent:
	\begin{enumerate}
		\item $ux$ and $vx$ are in different parts of the $k_x$-partition of $V(G_x)$.
		\item $ux$ and $vx$ are adjacent in $G_x$.
		\item $c(ux)c(vx)\in E(H)$.
		\item $x$ is not an obstruction of the path $(u,x,v)$.
		\item $(u,x,v)$ is an $H$-path in $G$.
	\end{enumerate}
\end{obs}

As a direct consequence of Observation~\ref{obs:equiv} and the definition of $H$-cycle, we have the following result.

\begin{obs}\label{obs:equiv4}
	Let $H$ be a graph possibly with loops, and $G$ be an $H$-colored graph, such that for every $x$ in $V(G)$, $G_x$ is a complete $k_x$-partite graph for some $k_x$ in $\mathbb{N}$. Suppose that $C=(u_1,\cdots,u_{n-1},u_n,u_1)$ is a cycle in $G$. The following statements are equivalent:
	\begin{enumerate}
		\item $C$ is an $H$-cycle in $G$.
		\item $(c(u_1 u_2),\cdots,c(u_{n-1} u_n),c(u_n u_1),c(u_1 u_2))$ is a walk in $H$.
		\item $u_1,\cdots,u_{n}$ are not obstructions of the cycle $C$.
		\item $u_{i-1}u_i$ and $u_{i+1}u_i$ are in different parts of the $k_{u_i}$-partition of $V(G_{u_i})$ for every $i$ in $\{1,\cdots,n\}$ (the subindices are taken modulo $n$).
	\end{enumerate}
\end{obs}

In~\cite{GRS1} was proved that, given any graphs $H$ and $G$, there exists an $H$-coloring of $G$ such that, for every $x$ in $V(G)$, either $G_x$ is a complete bipartite graph or $E(G_x)=\emptyset$.

In this paper, we work with $H$-colored complete graphs with restrictions given by the auxiliary graph $G_x$. The main results are the following: \\

\begin{teo}\label{teo:3Hciclo}
	Let $H$ be a graph possibly with loops and $G$ be an $H$-colored complete graph of order $n$, with $n\ge3$, such that for every $x$ in $V(G)$, $G_x$ is a complete $k_x$-partite graph for some $k_x$ in $\mathbb{N}$. Suppose that 
	\begin{enumerate}
		\item For every $x$ in $V(G)$, $k_x\ge \frac{n+1}{2}$,
		\item $G$ does not contain a cycle of length 4 with exactly 3 obstructions.
	\end{enumerate}
	Then each vertex of $G$ is contained in an $H$-cycle of length 3.
\end{teo}

\begin{teo}\label{teo:4Hciclo-1}
	Let $H$ be a graph possibly with loops and $G$ be an $H$-colored complete graph of order $n$, with $4\leq n<9$, such that for every $x$ in $V(G)$, $G_x$ is a complete $k_x$-partite graph for some $k_x$ in $\mathbb{N}$. Suppose that 
	\begin{enumerate}
		\item For every $x$ in $V(G)$, $k_x\ge \frac{n+1}{2}$, and
		\item $G$ does not contain a cycle of length 4 with exactly 3 obstructions.
	\end{enumerate}
	Then each vertex of $G$ is contained in an $H$-cycle of length 4.
\end{teo}

\begin{teo}\label{teo:4Hciclo-2}
	Let $H$ be a graph possibly with loops and $G$ be an $H$-colored complete graph of order $n$, with $n\ge9$, such that for every $x$ in $V(G)$, $G_x$ is a complete $k_x$-partite graph for some $k_x$ in $\mathbb{N}$. Suppose that 
	\begin{enumerate}
		\item For every $x$ in $V(G)$, $k_x\ge \frac{n+1}{2}$,
		\item $G$ does not contain a cycle of length 3 with exactly 2 obstructions, and
		\item $G$ does not contain a cycle of length 4 with exactly 3 obstructions.
	\end{enumerate}
	Then each vertex of $G$ is contained in an $H$-cycle of length~4.
\end{teo}

As a consequence of Theorems~\ref{teo:4Hciclo-1}~and~\ref{teo:4Hciclo-2} we have the following corollary.

\begin{cor}\label{cor:4Hciclo}
	
	Let $H$ be a graph possibly with loops and $G$ be an $H$-colored complete graph of order $n$, with $n\ge4$, such that for every $x$ in $V(G)$, $G_x$ is a complete $k_x$-partite graph for some $k_x$ in $\mathbb{N}$. Suppose that 
	\begin{enumerate}
		\item For every $x$ in $V(G)$, $k_x\ge \frac{n+1}{2}$,
		\item $G$ does not contain a cycle of length 3 with exactly 2 obstructions, and
		\item $G$ does not contain a cycle of length 4 with exactly 3 obstructions.
	\end{enumerate}
	Then each vertex of $G$ is contained in an $H$-cycle of length~4.
\end{cor}

Let $H$ be a graph possibly with loops and $G$ be an $H$-colored graph. We say that $G$ is $H$-pancyclic iff it contains an $H$-cycle of length $l$ for each $l$ in $\{3,\cdots,n\}$, and $G$ is said to be $H$-vertex pancyclic whenever each vertex of $G$ is contained in an $H$-cycle of length $l$ for every $l$ in $\{3,\cdots,n\}$.

The main contribution of this paper is to begin the study of $H$-vertex-pancyclicity for $H$-colored graphs. Although we are considering small cycle lengths in $H$-colored complete graphs, our investigations have underlined that these results are not easy to proof. As a matter of fact, several authors agree that, in the investigation of vertex-pancyclicity of a graph, the consideration of small cycle lengths is the most difficult part (see \cite{RST}). Moreover, these are the first essential steps in order to obtain a result implying the $H$-vertex-pancyclicity in an $H$-colored graph.

This paper is organized as follows: in section 2, we show the basic concepts and notation which will be useful in the developing of this work. In section 3, we exhibit some structural properties on the graph $G_x$, and we introduce the $H$-dependency property in an $H$-colored graph, which will be essential in order to prove the main results. In section 4, we prove the main theorems, and we exhibit explicit examples showing that the hypotheses of the first main theorem are tight. Finally, as a direct consequence of the main results, we obtain some classical theorems about the existence of properly colored cycles.

Given a digraph $D$ and a vertex $v$ in $V(D)$, we denote by $\delta^+(v)$ the out-degree of the vertex $v$, and by $\Delta^{+}(D)$ the maximum out-degree of $D$. 

We need the following result.

\begin{prop}[\cite{BJG}]\label{prop:torneoexgrad}
	If $T$ is a tournament of order $n$, with $n\ge1$, then $\Delta^{+}(T)\ge \frac{n-1}{2}$.
\end{prop}

\section{Terminology and Notation}

Let $G$ be a graph. In the rest of paper we will denote by: $N_G(u)$ the neighborhood and of $v$, $\delta_G(v)$ the degree of $v$, for $X\subseteq V(G)$, $G[X]$ the subgraph of $G$ induced by $X$, $G-X$ the subgraph of $G$ induced by $V(G)-X$, and if $X=\{a\}$, we write $G-a$ instead of $G-\{a\}$. If the graph $G$ is understood, we omit the subindex $G$.

A walk is a sequence $W=(v_0,v_1,\cdots,v_k)$ such that $v_i v_{i+1}\in E(G)$ for every $i$ in $\{0,1,\cdots,k-1\}$. The number $k$ is the length of $W$, denoted by $\ell(W)$. If $v_0=v_k$, then we say that $W$ is a closed walk. We say that the walk $W$ is a path iff $v_i\neq v_j$ for every $\{i,j\}$ subset of $\{0,1,\cdots,k\}$. A closed walk $(v_0,v_1,\cdots, v_k,v_0)$ is a cycle iff $k\ge 2$ and $(v_0,v_1,\cdots,v_k)$ is a path.

A subset $I$ of $V(G)$ is independent iff the subgraph $G[I]$ has no edges. For a fixed positive integer $k$, we say that a graph $G$ is a $k$-partite graph iff there exists a partition $\{V_1,\cdots,V_k\}$ of $V(G)$ where each $V_i$ is an independent set. Moreover, a $k$-partite graph with $\{V_1,\cdots,V_k\}$ a partition of $V(G)$ into independent sets, is said to be a complete $k$-partite graph iff for every $x$ in $A_i$ and for every $y$ in $A_j$, $x$ and $y$ are adjacent in $G$, for every $\{i,j\}$ subset of $\{1,\cdots,k\}$, with $i\neq j$.

\section{Previous results}

In order to make feasible the proof of the main theorems, we show some structural properties on the graph $G_x$ and we introduce some extra notation.

\begin{obs}\label{obs:kxinducida}
	Let $H$ be a graph possibly with loops, $G$ be an $H$-colored graph without isolated vertices, and $D$ an induced (by $V(D)$) subgraph of $G$. If for every $x$ in $V(G)$, $G_x$ is a complete $k_x$-partite graph for some $k_x$ in $\mathbb{N}$, then for every $x$ in $V(D)$, $D_x$ is a complete $l_x$-partite graph for some $l_x$ in $\mathbb{N}$. Moreover, if $\{P_1^x,P_2^x,\cdots ,P_{k_x}^x\}$ is the $k_x$-partition of $V(G_x)$ into independent sets, then $\{P_i^x\cap V(D_x): P_i^x\cap V(D_x)\neq \emptyset, i\in\{1,2,\cdots, k_x\}\}$ is the $l_x$-partition of $V(D_x)$ into independent sets.
\end{obs}

\begin{notac}\label{not:lxd}
	If $D$ is an induced subgraph of $G$ without isolated vertices, then for every $x$ in $V(D)$ we write $l_x^D$ instead of $l_x$, where $l_x$ is the one referred in Observation~\ref{obs:kxinducida}.
\end{notac}

From now on, $H$ is a, possibly with loops, graph and $G$ is an $H$-colored complete graph with $c:E(G)\longrightarrow V(H)$ a predetermined $H$-coloring of $G$.

\begin{notac}\label{notac:part}
	Whenever $G_x$ is a complete $k_x$-partite graph for some $k_x$ in $\mathbb{N}$, we will denote by $\{P_1^x,P_2^x,\cdots,P_{k_x}^x\}$ the $k_x$-partition of $V(G_x)$ into independent sets.
\end{notac}

Let $A$ be a subset of $V(G)$ and $v$ be a vertex in $V(G)-A$. We say that $A$ has the $H$-dependence property with respect to the vertex $v$ iff for every $\{a,a'\}$ subset of $A$, $a$ is an obstruction of the walk $(v,a,a')$, or $a'$ is an obstruction of the walk $(a,a',v)$. The following lemma follows directly from this definition.

\begin{lem}\label{lema:subchprop}
	Let $A$ be a subset of $V(G)$, $A'$ a subset of $A$ and $v\in V(G)-A$. If $A$ has the $H$-dependence property with respect to the vertex $v$, then $A'$ has the $H$-dependence property with respect to the vertex $v$.
\end{lem}

\begin{prop}\label{prop:hdep}
	Suppose that for every $x$ in $V(G)$, $G_x$ is a complete $k_x$-partite graph for some $k_x$ in $\mathbb{N}$. Let $A$ be a subset of $V(G)$ and $v$ be a vertex in $V(G)-A$. If $A$ has the $H$-dependence property with respect to the vertex $v$, then there exists some vertex $a$ in $A$ such that
	\begin{enumerate}
		\item $l_a^D\le \frac{|A|+1}{2}$, where $D=G[A]$, and
		\item if $|A|\ge 2$, then $a$ is an obstruction of the walk $(v,a,a')$ for some $a'$ in $N_D(a)$.
	\end{enumerate}
\end{prop}

\begin{proof}
	\begin{description}
		\item[(1)] Suppose that $|A|=m$ for some $m$ in $\mathbb{N}$, and give an orientation of $E(D)$ as follows: for every $\{x,y\}$ subset of $A$, if $x$ is an obstruction of the walk $(v,x,y)$, then orient the edge from $x$ to $y$. If $x$ is an obstruction of $(v,x,y)$ and $y$ is an obstruction of the walk $(v,y,x)$, then orient the edge $xy$ arbitrarily. This orientation of the graph $D$ is well-defined since $A$ has the $H$-dependence property with respect to the vertex $v$. Let $D'$ be such orientation of the graph $D$. Notice that $D'$ is a tournament of order $m$, hence by Proposition~\ref{prop:torneoexgrad}, $\Delta^{+}(D')\ge \frac{m-1}{2}$, that is, there exists a vertex $u$ in $A$ such that $\delta^{+}(u)=\Delta^{+}(D')\ge \frac{m-1}{2}$. Suppose that $N^{+}_{_{D'}}(u)=\{x_1,x_2,\cdots,x_s\}$, where $s\ge \frac{m-1}{2}$. By the construction of the tournament $D'$, we have that for every $i$ in $\{1,2,\cdots,s\}$, $u$ is an obstruction of the walk $(v,u,x_i)$, which implies by~Observation~\ref{obs:equiv} that $ux_i$ and $uv$ are in the same set of the $k_u$-partition of $V(G_u)$. Hence, by Observation~\ref{obs:kxinducida}, $ux_1, ux_2,\cdots,ux_s$ are in the same set of the $l_u^D$-partition of $V(D_u)$, where $s\ge \frac{m-1}{2}$.\\
		Finally, as $\delta_{_{D}}(u)=m-1$, and $$m-1-s+1\le m-1-\frac{m-1}{2}+1=\frac{2m-m+1}{2}=\frac{m+1}{2},$$ we have that there exists at most $\frac{m+1}{2}$ parts in the $l_u^D$-partition of $V(D_u)$. Therefore $l_u^D\leq \frac{m+1}{2}=\frac{|A|+1}{2}$.
		
		\item[(2)] Let $A$ be a subset of $V(G)$ and $v$ be in $V(G)-A$. We want to prove that, if $A$ has the $H$-dependence property with respect to the vertex $v$, then there exists a vertex $a$ in $A$ such that $l_{a}^D\leq \frac{|A|+1}{2}$, where $D=G[A]$, and moreover, if $|A|\ge 2$, then $a$ is an obstruction of $(v,a,a')$ for some $a'$ in $N_D(a)$. We proceed by induction on $|A|$, with $|A|\geq 2$.\\
		First suppose that $|A|=2$, say $A=\{x,y\}$. Since $A$ has the $H$-dependence property with respect to the vertex $v$, hence $x$ is an obstruction of the walk $(v,x,y)$ or $y$ is an obstruction of the walk $(v,y,x)$. Suppose without loss of generality that $x$ is an obstruction of the walk $(v,x,y)$, thus $x$ is the desired vertex, since $l_{x}^D=1\le \frac{3}{2}=\frac{|A|+1}{2}$, where $D=G[A]$.\\
		
		Now suppose that, if $A'$ is a subset of $V(G)$ and $v$ is a vertex in $V(G)-A'$, such that $A'$ has the $H$-dependency property with respect to the vertex $v$ and $|A'|\le m-1$, then there exists a vertex $a'$ in $A'$ such that $l_{a'}^{D'}\leq \frac{|A'|+1}{2}$, where $D'=G[A']$, and moreover, if $|A'|\ge 2$, then $a'$ is an obstruction of the walk $(v,a',b)$ for some $b$ in $N_{D'}(a')$.\\
		
		Now take a set $A$ having the $H$-dependency property with respect to the vertex $v$ and $|A|=m$, with $m> 2$. By the proof of the part (1), we have that there exists a vertex $a$ in $A$ such that $l_a^D\le \frac{|A|+1}{2}$, where $D=G[A]$. If $a$ is an obstruction of the walk $(v,a,a')$ for some $a'$ in $N_D(a)$, then $a$ is the desired vertex. Suppose now that $a$ is not an obstruction of the walk $(v,a,a')$ for every $a'$ in $N_D(a)$.\\
		Let $A'=A\setminus\{a\}$ be. By Lemma~\ref{lema:subchprop}, $A'$ has the $H$-dependency property with respect to the vertex $v$. Moreover, given that $|A'|=m-1\ge 2$, the inductive hypothesis implies that there exists a vertex $a'$ in $A'$ such that $l_{a'}^{D'}\leq \frac{|A'|+1}{2}=\frac{m}{2}$, where $D'=G[A']$ and such that $a'$ is an obstruction of the walk $(v,a',b)$ for some $b$ in $N_{D'}(a')$. Since $a'$ is an obstruction of the walk $(v,a',b)$, we have by Observation~\ref{obs:equiv} that $va'$ and $a'b$ are in the same part of the $k_{a'}$-partition of $V(G_{a'})$. Now, as $A$ has the $H$-dependency property with respect to the vertex $v$ and $\{a,a'\}$ is a subset of $A$, it follows that $a$ is an obstruction of $(v,a,a')$ or $a'$ is an obstruction of $(a,a',v)$. However, by our assumption about $a$, we have that $a'$ is an obstruction of the walk $(a,a',v)$, which implies by Observation~\ref{obs:equiv} that $aa'$ and $va'$ are in the same part of the $k_{a'}$-partition of $V(G_{a'})$. Thus, $aa'$ and $a'b$ are in the same part of the $k_{a'}$-partition of $V(G_{a'})$, in particular, by Observation~\ref{obs:kxinducida} we have that $aa'$ and $a'b$ are in the same part of the $l_{a'}^{D}$-partition of $V(D_{a'})$. In addition, we have that $l_{a'}^{D'}\leq \frac{|A'|+1}{2}=\frac{m}{2}$, so $l_{a'}^{D}=l_{a'}^{D'}\leq \frac{m}{2}\le\frac{m+1}{2}= \frac{|A|+1}{2}$, and therefore $a'$ is the desired vertex.
		
	\end{description}
\end{proof}

\section{Main Results}

\textbf{Proof of Theorem~\ref{teo:3Hciclo}.}

Proceeding by contradiction, suppose that there exists a vertex $v$ in $V(G)$ such that $v$ is not contained in an $H$-cycle of length 3 in $G$. As $G_v$ is a complete $k_v$-partite graph for some $k_v$ in $\mathbb{N}$, by Notation~\ref{notac:part}, $\{P_1^v,P_2^v,\cdots,P_{k_v}^v\}$ is the $k_v$-partition of $V(G_v)$ into independent sets. Now, for every $i$ in $\{1,2,\cdots,k_v\}$, we define $N_i^v=\{u\in N_G(v):uv\in P_i^v\}$.\\

\begin{obs}\label{obs:obstciclo}
	For every $\{i,j\}$ subset of $\{1,2,\cdots,k_v\}$, for each $x$ in $N_i^v$ and for any $y$ in $N_j^v$, we have that:
	\begin{description}
		\item[(1)] $v$ is not an obstruction of the cycle $(v,x,y,v)$, and
		\item[(2)] $x$ is an obstruction of the walk $(v,x,y)$ or $y$ is an obstruction of the walk $(v,y,x)$.
	\end{description} 
\end{obs}

\begin{proof}
	\begin{description}
		\item[(1)] By construction of the sets $N_i^v$ and $N_j^v$, it follows that $vx$ and $vy$ are in different parts of the $k_v$-partition of $V(G_v)$, which implies by Observation~\ref{obs:equiv} that $v$ is not an obstruction of the path $(x,v,y)$, that is, $v$ is not an obstruction of the cycle $(v,x,y,v)$.
		\item[(2)] Proceeding by contradiction, it follows directly from Observations~\ref{obs:equiv4} and \ref{obs:obstciclo}(1) that $(v,x,y,v)$ is an $H$-cycle of length~3 in $G$ containing $v$, which is impossible by our assumption. 
	\end{description}
\end{proof}

Renaming (if necessary) we can assume that $|N_1^v|=|N_2^v|=\cdots =|N_r^v|=1$ and $2\le |N_{r+1}^v|\le \cdots \le |N_{k_v}^v|$ for some nonnegative integer $r$.\\

\textbf{Claim 1.} $r\ge 2$. 

\begin{proof}
	Otherwise, if $r\leq 1$, then $$n-1=|N_1^v|+\cdots+|N_r^v|+|N_{r+1}^v|+\cdots+|N_{k_v}^v|\ge r+2(k_{v}-r)=2k_{v}-r\ge 2k_{v}-1$$which implies that $k_v\le \frac{n}{2}$. This is impossible since by hypothesis $k_v\ge \frac{n+1}{2}$.
\end{proof}

Consider $D=G[N_1^v\cup N_2^v\cup \cdots \cup N_r^v]$, $D'=G[N_{r+1}^v\cup\cdots\cup N_{k_v}^v]$ and $s=k_v-r$.\\

Given that $|N_1^v|=|N_2^v|=\cdots =|N_r^v|=1$, it follows directly from Observation~\ref{obs:obstciclo}(2) the following claim.\\

\textbf{Claim 2.} $V(D)$ has the $H$-dependency property with respect to the vertex~$v$.\\

Now the $H$-dependency property of $V(D)$ (with respect to the vertex~$v$) and Proposition~\ref{prop:hdep}, imply that there exists a vertex $a$ in $V(D)$ such that $l_a^{D}\le \frac{|V(D)|+1}{2}=\frac{r+1}{2}$, and moreover, since $r\ge 2$, it follows that $a$ is an obstruction of the walk $(v,a,a')$ for some $a'$ in $N_{D}(a)$. By construction, we know that $N_i^v=\{a\}$ for some $i$ in $\{1,2,\cdots,r\}$.\\

\textbf{Claim 3.} $k_a\leq l_a^{D}+s$ (see~Notation~\ref{not:lxd}).
\begin{proof}
	Recall that $G_a$ is a complete $k_a$-partite graph for some $k_a$ in $\mathbb{N}$, and $\mathcal{P}=\{P_1^{a},P_2^{a},\cdots, P_{k_a}^{a}\}$ is the $k_a$-partition of $V(G_a)$ into independent sets. By the Observation~\ref{obs:kxinducida} and given that $D$ is an induced subgraph of $G$, it follows that $\mathcal{Q}=\{P_i^{a}\cap V(D_a): P_i^{a}\cap V(D_a)\neq \emptyset\, i\in \{1,2,\cdots, k_a\}\}$ is the $l_{a}^D$-partition of $V(D_a)$ into independent sets, where $V(D_a)=\{e\in E(D):e\ \mbox{is incident with}\ a\}$. Renaming, if necessary, we say that $\mathcal{Q}=\{P_1^{a}\cap V(D_a),P_2^{a}\cap V(D_a),\cdots, P_{l_a^D}^{a}\cap V(D_a)\}$ is the $l_a^D$-partition of $V(D_a)$ into nonempty independent sets.\\
	Now, in order to prove this upper bound of $k_a=k_a^{G}$ in terms of $l_a^{D}$, consider the edges incident with $a$ that are not vertices in $D_a$, that is, $V(G_a)-V(D_a)=\{va\}\cup \{az:z\in N_i^v, i\in \{r+1,\cdots,k_v\}\}$. Since $a$ is an obstruction of the walk $(v,a,a')$ for some $a'$ in $N_{D}(a)$, thus $va$ and $aa'$ are in the same part of the $k_a$-partition of $V(G_a)$ by the Observation~\ref{obs:equiv}. Moreover, as $aa'$ is in $V(D_a)$, it follows that there exists $\alpha$ in $\{1,2,\cdots,l_a^D\}$ such that $aa'\in P_{\alpha}^{a}$, so $\{aa',va\}\subseteq P_{\alpha}^{a}$, where $P_{\alpha}^{a}$ is already considered in the first $l_a^D$ parts of $\mathcal{P}$. Hence, when we consider the edge $va$, we add no extra part in the $k_a$-partition of $V(G_a)$ with respect to the number of parts of the $l_a^D$-partition of $V(D_a)$.\\
	Next, we will prove that for every $j$ in $\{r+1,\cdots,k_v\}$, we add at most one part in $\mathcal{P}$ with respect to the number of parts in $\mathcal{Q}$. \\
	
	Let $j$ be an element of $\{r+1,\cdots,k_v\}$ and $x$ be a vertex in $N_j^v$.
	
	\begin{description}
		\item[\textbf{Case 1.}] $a$ is an obstruction of the walk $(v,a,x)$.\\
		If $a$ is an obstruction of the walk $(v,a,x)$, then by Observation~\ref{obs:equiv}, $va$ and $ax$ are in the same set of the $k_a$-partition of $V(G_a)$, and given that $va$ is in $P_{\alpha}^{a}$, thus $ax\in P_{\alpha}^{a}$, where $P_{\alpha}^{a}$ is already counted in the first $l_a^D$ parts of $\mathcal{P}$. Therefore, in this case we add no extra part in $\mathcal{P}$ with respect to the number of parts in $\mathcal{Q}$.
		
		\item[\textbf{Case 2.}] $a$ is not an obstruction of the walk $(v,a,x)$.\\
		If $a$ is not an obstruction of the walk $(v,a,x)$, then $x$ is an obstruction of the walk $(v,x,a)$, by Observation~\ref{obs:obstciclo}. Moreover, if $a$ is not an obstruction of the walk $(v,a,x)$, then by Observation~\ref{obs:equiv}, $va$ and $ax$ are in different sets of the $k_a$-partition of $V(G_a)$, with $va\in P_{\alpha}^{a}$. We can assume that $ax$ is in $P_{\beta}^{a}$ for some $\beta$ in $\{l_a^D +1,\cdots,k_a\}$, otherwise, when we consider the edge $ax$, it is already counted in the first $l_a^D$ parts of $\mathcal{P}$.\\
		
		\textbf{Claim 3.1.} For every $y$ in $N_j^v$, $ay\in P_{\alpha}^{a}$ or $ay\in P_{\beta}^{a}$.			
		\begin{proof}
			Let $y$ be in $N_j^v$, with $y\neq x$.
			
			\begin{description}
				\item[Subcase 2.1] $a$ is an obstruction of the walk $(v,a,y)$.\\
				Proceeding in a similar way as in Case~1, we can prove that $va$ and $ay$ are in the same set of the $k_a$-partition of $V(G_a)$, and given that $va$ is in $P_{\alpha}^{a}$, we have that $ay\in P_{\alpha}^{a}$, as desired.
				\item[Subase 2.2] $a$ is not an obstruction of the walk $(v,a,y)$.\\
				Since $a$ is not an obstruction of the walk $(v,a,y)$, it follows that $y$ is an obstruction of the walk $(v,y,a)$ by Observation~\ref{obs:obstciclo}, and recall that $x$ is an obstruction of the walk $(v,x,a)$. On the other hand, given that $\{vx,vy\}$ is a subset of $N_j^v$, we have by construction of the set $N_j^v$ that $vx$ and $vy$ are in the same part of the $k_v$-partition of $V(G_v)$, which implies that $v$ is an obstruction of the walk $(x,v,y)$, by Observation~\ref{obs:equiv}. Hence, $y$, $x$ and $v$ are obstructions of the cycle $C=(v,x,a,y,v)$, with $\ell(C)=4$, so it follows by the hypothesis~2 of Theorem~\ref{teo:3Hciclo} that $a$ is also an obstruction of the cycle $C$. Since $a$ is an obstruction of the cycle $C$, we have by Observation~\ref{obs:equiv} that $ay$ and $ax$ are in the same part of the $k_a$-partition of $V(G_a)$, and given that $ax\in P_{\beta}^{a}$, thus $ay\in P_{\beta}^{a}$, as desired.
			\end{description}
			
			So Claim~3.1 is proved.
		\end{proof}
		Since for every $y$ in $N_j^v$, $ay\in P_{\alpha}^{a}$ with $\alpha$ in $\{1,2,\cdots,l_a^D\}$, or $ay\in P_{\beta}^{a}$ with $\beta$ in $\{l_a^D +1,\cdots, k_a\}$, we have that for every $j$ in $\{r+1,\cdots ,k_v\}$ we add at most one part in $\mathcal{P}$ with respect to the number of parts in $\mathcal{Q}$ (namely $P_{\beta}^{a}$), where $|\mathcal{P}|=|\{P_1^{a},P_2^{a},\cdots,P_{k_a}^{a}\}|=k_a$, $|\mathcal{Q}|=|\{P_1^{a}\cap V(D_a),P_2^{a}\cap V(D_a),\cdots, P_{l_a^D}^{a}\cap V(D_a)\}|=l_a^D$ and $s=k_v -r$.
	\end{description}
	Therefore $k_a\le l_a^D +s$.		
\end{proof}

Now, $|V(D')|=|N_{r+1}^v|+\cdots+|N_{k_v}^v|\ge 2(k_v -r)=2s$, so $s\le \frac{|V(D')|}{2}$. Hence, $$k_a\le l_a^D+s\le \frac{r+1}{2}+\frac{|V(D')|}{2}=\frac{(|V(D)|+1)+|V(D')|}{2}=\frac{n}{2},$$which is impossible since the hypothesis~1 of Theorem~\ref{teo:3Hciclo} establishes that $k_a\ge \frac{n+1}{2}$.\\

Therefore $v$ is contained in an $H$-cycle of length 3 in $G$. $\square$\\

The two following observations show examples of $H$-colored complete graphs, showing that the hypotheses of Theorem~\ref{teo:3Hciclo} are tight.

\begin{obs}
	The hypothesis 1 of Theorem~\ref{teo:3Hciclo} cannot be dropped. Notice that in the $H$-colored complete graph $G$ of the Figure~\ref{fig:ejemplo}, we have that for every $x$ in $V(G)$, $G_x$ is a complete 2-partite graph, where $2<\frac{|V(G)|+1}{2}$, thus the hypothesis 1 of Theorem~\ref{teo:3Hciclo} is not fulfilled. Also, the cycle $(a,b,c,d,a)$ has no obstructions, the cycle $(a,b,d,c,a)$ has 2 obstructions (namely $a$ and $c$), and the cycle $(a,d,b,c,a)$ has 2 obstructions (namely $b$ and $d$), therefore the hypothesis~2 of Theorem~\ref{teo:3Hciclo} is satisfied. Nevertheless, the vertex $a$ in $V(G)$ is not in an $H$-cycle of length 3 in $G$, since $a$ is an obstruction of the cycle $(a,b,c,a)$, $c$ is an obstruction of the cycle $(a,c,d,a)$, and $d$ is an obstruction of the cycle $(a,b,d,a)$.
\end{obs}

\begin{figure}[h!]
	\begin{center}
		\includegraphics[scale=0.4]{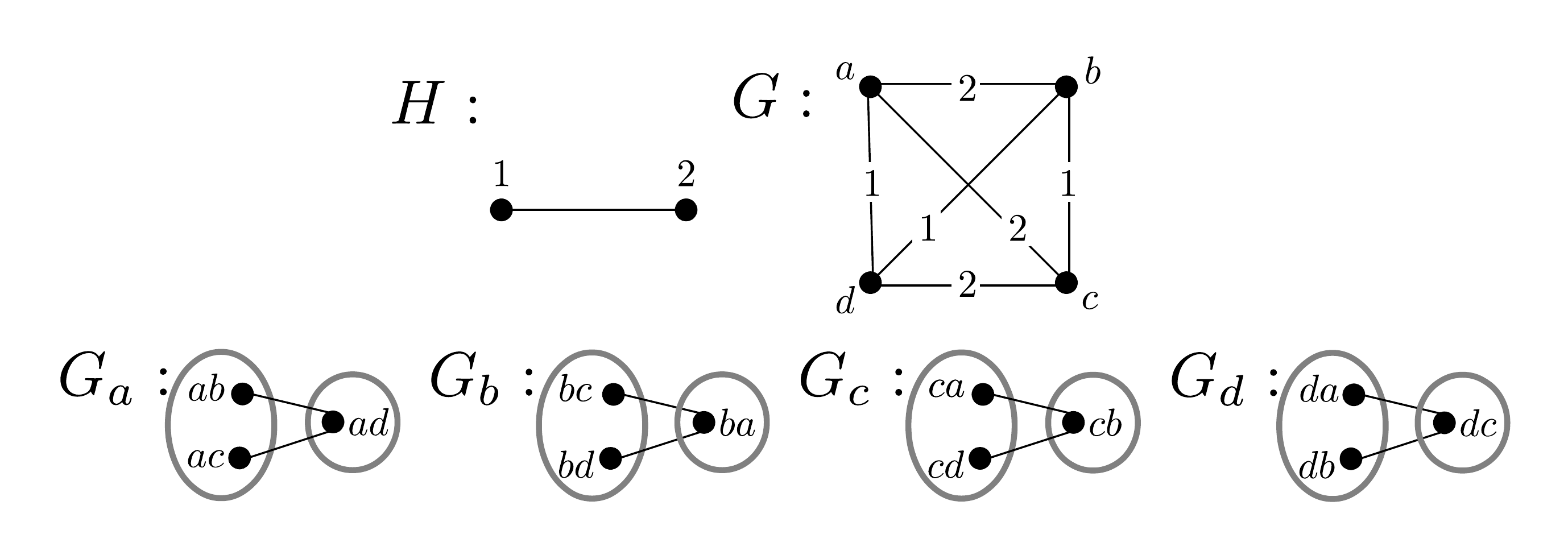}
		\caption{The hypothesis~1 of Theorem~\ref{teo:3Hciclo} cannot be dropped.}
		\label{fig:ejemplo}
	\end{center}
\end{figure}

\begin{obs}
	The hypothesis 2 of Theorem~\ref{teo:3Hciclo} cannot be dropped as the $H$-colored complete graph in the Figure~\ref{fig:K7} shows. Notice that for every $a$ in $V(G)$, $G_a$ is a complete 4-partite graph, that is, the hypothesis~1 of Theorem~\ref{teo:3Hciclo} is fulfilled; and the vertex $r$ in $V(G)$ is not contained in an $H$-cycle of length 3, since each one of the 15 cycles of length 3 containing $r$ has the indicated obstruction: $(r,x,w,r)$, obstruction in vertex $w$; $(r,x,v,r)$, obstruction in vertex $x$; $(r,x,u,r)$, obstruction in vertex $x$; $(r,x,t,r)$, obstruction in vertex $t$; $(r,x,s,r)$, obstruction in vertex $s$; $(r,w,v,r)$, obstruction in vertex $w$; $(r,w,u,r)$, obstruction in vertex $u$; $(r,w,t,r)$, obstruction in vertex $t$; $(r,w,s,r)$, obstruction in vertex $s$; $(r,w,u,r)$, obstruction in vertex $u$; $(r,v,u,r)$, obstruction in vertex $u$; $(r,v,t,r)$, obstruction in vertex $v$; $(r,v,s,r)$, obstruction in vertex $v$; $(r,u,t,r)$ obstruction in vertex $r$; $(r,u,s,r)$ obstruction in vertex $r$; $(r,t,s,r)$, obstruction in vertex $r$.	However, the hypothesis~2 of Theorem~\ref{teo:3Hciclo} is not satisfied because $(r,s,x,t,r)$ is a cycle of length 4 in $G$, with exactly 3 obstructions (namely $r,s$ and $t$).
	
	\begin{figure}[h!]
		\begin{center}
			\includegraphics[scale=0.3]{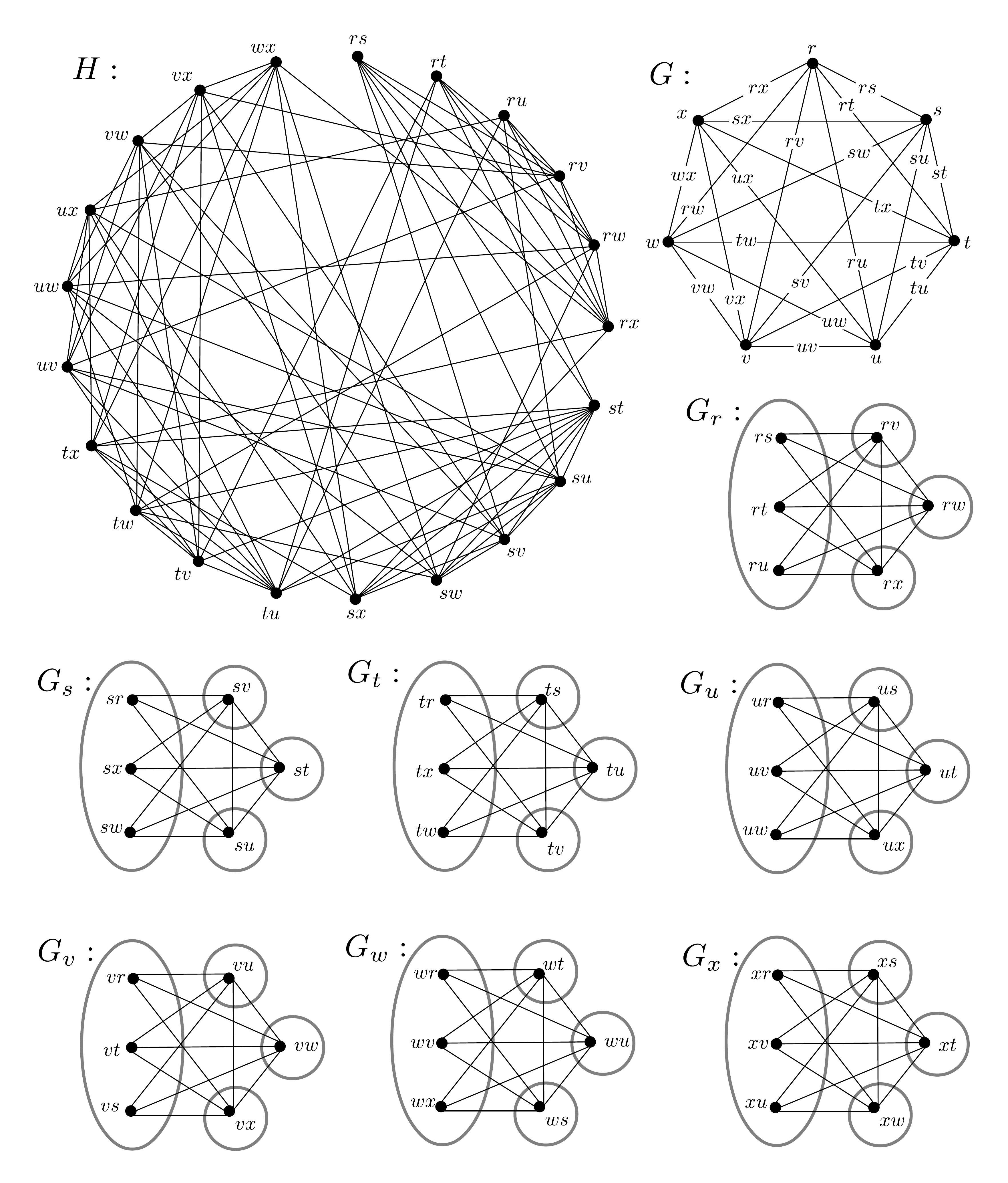}
			\caption{The hypothesis~2 of Theorem~\ref{teo:3Hciclo} cannot be dropped.}
			\label{fig:K7}
		\end{center}
	\end{figure}
\end{obs}

\textbf{Proof of Theorem~\ref{teo:4Hciclo-1}}

If $n=4$, then for every $x$ in $V(G)$ we have that $k_x\ge \frac{5}{2}$, and $k_x\le \delta(x)=3$, so $k_x=3$. And hence, for every two edges $a$ and $b$ incident with $x$, we have that $a$ and $b$ are in different parts of the $k_x$-partition of $V(G_x)$.\\

\hspace*{0.5cm}If $n=5$, then for every $x$ in $V(G)$ we have that $k_x\ge 3$, and $k_x\le \delta (x)=4$, so $k_x=3$ or $k_x=4$. Then, for every vertex $x$ in $V(G)$, at most 2 edges incident with $x$ are in the same part of the $k_x$-partition of $V(G_x)$. As a consequence, for each vertex $x$ in $V(G)$, we know that there are at most two cycles of length 4 in $G$ passing through $x$ that are not $H$-cycles. And we are done.\\

\hspace*{0.5cm}The case $n=6$, can be easily carried as a consequence of the case $n=5$, because for each $x$ in $V(G)$, $G-x$ satisfies the hypotheses of Theorem~\ref{teo:4Hciclo-1}.\\

\hspace*{0.5cm}If $n=7$, then for every $x$ in $V(G)$, we have that $k_x\ge 4$. Proceeding by contradiction, suppose there exists a vertex $v$ in $V(G)$ such that $v$ is not contained in an $H$-cycle of length 4 in $G$. It follows by Theorem~\ref{teo:3Hciclo} that there exists an $H$-cycle of length 3 containing $v$, say $(v,x,y,v)$. Thus, by Observation~\ref{obs:equiv}, we have that $yv$ and $xy$ are in different parts of the $k_v$-partition of $V(G_v)$, $vx$ and $yx$ are in different parts of the $k_x$-partition of $V(G_x)$, and $vy$ and $xy$ are in different parts of the $k_y$-partition of $V(G_y)$. Without loss of generality, suppose that $xv\in P_1^v$ and $yv\in P_2^v$, $vx\in P_1^x$ and $yx\in P_2^x$, and $vy\in P_1^y$ and $xy\in P_2^y$. As for every $a$ in $V(G)$, $k_a\ge 4$, we have that for every $t$ in $\{v,x,y\}$, there exists at least two different parts of the $k_t$-partition of $V(G_t)$, other than the two already mentioned. Since there are only 4 vertices in $V(G)-\{v,x,y\}$, it follows that there is a vertex $w$ such that, without loss of generality, $wx\in P_3^x$ and $wy\in P_3^y$, and suppose that $xw\in P_1^w$.\\
\textbf{Claim 1.} $\{vw,xw,yw\}$ is a subset of $P_1^{w}$, and for every $t$ in $V(G)-\{v,w,x,y\}$, $tw\notin P_1^{w}$.
\begin{proof}
	First we prove that $yw$ is in $P_1^w$. Proceeding by contradiction, suppose that $yw\notin P_1^w$, then by Observation~\ref{obs:equiv4}, $(v,y,w,x,v)$ is an $H$-cycle of length 4 in $G$ containing $v$ (since $vy\in P_2^v\cap P_1^y$, $yw\in P_3^y-P_1^w$, $wx\in P_3^y\cap P_1^w$ and $xv\in P_1^v\cap P_1^x$), which is impossible by our assumption. Therefore $yw\in P_1^w$.\\
	Now, we prove that $vw\in P_1^w$. Proceeding by contradiction, suppose that $vw\notin P_1^w$.\\
	\hspace*{0.5cm}When $wv\in P_1^v$, we have that by Observation~\ref{obs:equiv4},  $(y,v,w,x,y)$ is an $H$-cycle of length 4 in $G$ containing $v$ (as $yv\in P_2^v\cap P_1^y$, $vw\in P_1^v-P_1^w$, $wx\in P_3^x\cap P_1^w$ and $xy\in P_2^y\cap P_2^x$), a contradiction. \\
	\hspace*{0.5cm}When $wv\notin P_1^v$, the previous observation implies that $(y,w,v,x,y)$ is an $H$-cycle of length 4 in $G$ containing $v$ (notice that $yw\in P_3^y\cap P_1^w$, $wv\notin P_1^w\cup P_1^v$, $vx\in P_1^v\cap P_1^x$ and $xy\in P_2^y\cap P_2^x$), which is impossible.\\ 
	We conclude that $vw\in P_1^w$. \\
	Therefore $\{vw,xw,yw\}\subseteq P_1^w$, and  moreover, as $G_w$ is a complete $k_w$-partite graph, with $k_w\ge 4$, it follows that for each $t$ in $V(G)-\{v,w,x,y\}$, $tw\notin P_1^w$.
\end{proof}

To conclude the proof when $n=7$, we consider the two following cases. Both cases lead us to a contradiction.

\begin{description}
	\item[\textbf{Case 1.}] $wv\notin P_1^v \cup P_2^v$.\\
	Without loss of generality, suppose that $wv\in P_3^v$. Given that $G_v$ is a complete $k_v$-partite graph, with $k_v\ge 4$, consider $u$ in $V(G)-\{v,w,x,y\}$ such that $vu\in P_4^v$; and since $u\in V(G)-\{v,x,y,w\}$, it follows by Claim~1 that $uw\notin P_1^w$. Suppose that $wu\in P_1^{u}$.\\
	\textbf{Claim~2.} $\{vu,wu,xu,yu\}\subseteq P_1^{u}$.
	\begin{proof}
		As $wu\in P_1^{u}$, there remains to prove that $\{vu,xu,yu\}\subseteq P_1^{u}$.\\
		First we prove that $vu\in P_1^{u}$. Proceeding by contradiction, suppose that $vu\notin P_1^{u}$, hence by Observation~\ref{obs:equiv4}, $(x,v,u,w,x)$ is an $H$-cycle of length 4 in $G$ containing $v$ (since $xv\in P_1^v\cap P_1^x$, $vu\in P_4^v-P_1^{u}$, $uw\in P_1^{u}-P_1^w$ and $wx\in P_3^x\cap P_1^w$), which is impossible by our assumption. Therefore $vu\in P_1^{u}$.\\
		
		Next, we prove that $xu\in P_1^{u}$. Proceeding by contradiction, suppose that $xu\notin P_1^{u}$. \\
		\hspace*{0.5cm}When $ux\in P_2^x$, Observation~\ref{obs:equiv4}, $wv\in P_1^w\cap P_3^v$, $vx\in P_1^v\cap P_1^x$, $xu\in P_2^x-P_1^{u}$ and $uw\in P_1^{u}-P_1^w$ imply that $(w,v,x,u,w)$ is an $H$-cycle of length 4 in $G$ containing $v$, which is impossible. \\
		\hspace*{0.5cm}When $ux\notin P_2^x$, the previous argument implies that $(v,y,x,u,v)$ is an $H$-cycle of length 4 in $G$ containing $v$ (since $vy\in P_2^v\cap P_1^y$, $yx\in P_2^y\cap P_2^x$, $xu\notin P_2^x\cup P_1^{u}$ and $uv\in P_1^{u}\cap P_4^v$), a contradiction.\\ 
		We conclude that $xu\in P_1^{u}$.\\
		
		Now, we prove that $yu\in P_1^{u}$. Proceeding by contradiction, suppose that $yu\notin P_1^{u}$. \\
		\hspace*{0.5cm}When $uy\in P_2^y$, Observation~\ref{obs:equiv4}, $wv\in P_1^w\cap P_3^v$, $vy\in P_2^v\cap P_1^y$, $yu\in P_2^y-P_1^{u}$ and $uw\in P_1^{u}-P_1^w$ imply that $(w,v,y,u,w)$ is an $H$-cycle of length 4 in $G$ containing $v$, which is impossible. \\
		\hspace*{0.5cm}When $uy\notin P_2^y$, it follows from the previous observation that $(v,x,y,u,v)$ is an $H$-cycle of length 4 in $G$ containing $v$ (as $vx\in P_1^v\cap P_1^x$, $xy\in P_2^y\cap P_2^x$, $yu\notin P_2^y\cup P_1^{u}$ and $uv\in P_1^{u}\cap P_4^v$), a contradiction.\\ 
		We conclude that $yu\in P_1^{u}$.\\
		
		This concludes the proof of Claim~2. Hence, $\{vu,wu,xu,yu\}\subseteq P_1^{u}$.
	\end{proof}
	Finally, as $\{vu,wu,xu,yu\}$ is a subset of $P_1^{u}$ and $\delta_G(u)=6$, we have that $k_u\le 3$, which is impossible since by hypothesis $k_u\ge 4$.
	\item[\textbf{Case 2.}] $wv\in P_1^v\cup P_2^v$.\\
	By symmetry, suppose without loss of generality that $wv\in P_1^v$, and since $G_v$ is a complete $k_v$-partite graph, with $k_v\ge 4$, we can consider a subset $\{u,s\}$ of $V(G)-\{v,w,x,y\}$ such that $vu\in P_3^v$ and $vs\in P_4^v$. Given that $\{u,s\}$ is a subset of $V(G)-\{v,x,y,w\}$, it follows from Claim~1 that $uw\notin P_1^w$ and $sw\notin P_1^w$. Suppose without loss of generality that $wu\in P_1^{u}$ and $ws\in P_1^s$.\\
	\textbf{Claim~3.} $\{vu,wu,yu\}\subseteq P_1^{u}$.
	\begin{proof}
		Since $wu\in P_1^{u}$, it suffices to prove that $\{vu,yu\}\subseteq P_1^{u}$.\\
		First, we prove that $vu\in P_1^{u}$. Proceeding by contradiction, suppose that $vu\notin P_1^{u}$, thus by Observation~\ref{obs:equiv4} and given that $yv\in P_2^v\cap P_1^y$, $vu\in P_3^v-P_1^{u}$, $uw\in P_1^{u}-P_1^w$ y $wy\in P_1^w\cap P_3^y$, we have that $(y,v,u,w,y)$ is an $H$-cycle of length 4 in $G$ containing $v$, which is impossible. Therefore $vu\in P_1^{u}$.\\
		
		Now, we prove that $yu\in P_1^{u}$. Proceeding by contradiction, suppose that $yu\notin P_1^{u}$.\\
		\hspace*{0.5cm}When $uy\in P_2^y$, Observation~\ref{obs:equiv4}, $vy\in P_2^{v}\cap P_1^y$, $yu\in P_2^y-P_1^{u}$, $uw\in P_1^{u}- P_1^{w}$ and $wv\in P_1^{v}\cap P_1^w$ imply that $(v,y,u,w,v)$ is an $H$-cycle of length 4 in $G$ containing $v$, a contradiction. \\
		\hspace*{0.5cm}When $uy\notin P_2^y$, we have from the previous observation that $(v,u,y,x,v)$ is an $H$-cycle of length 4 in $G$ containing $v$ (as $vu\in P_1^{u}\cap P_3^v$, $uy\notin P_1^{u}\cup P_2^y$, $yx\in P_2^y\cap P_2^x$ and $xv\in P_1^v\cap P_1^x$), which is impossible.\\ 
		We conclude that $xu\in P_1^{u}$.
	\end{proof}
	
	Proceeding in a very similar way to the proof of Claim~3 (taking $s$ instead of $u$), we can prove that:
	
	\textbf{Claim 4.} $\{vs,ws,ys\}\subseteq P_1^{s}$.
	\vspace*{0.3cm}
	
	\textbf{Claim 5.} $us\in P_1^{u}$ or $us\in P_1^{s}$.
	\begin{proof}
		Otherwise, if $us\notin P_1^{u}$ and $us\notin P_1^{s}$, then by Observation~\ref{obs:equiv4} and given that $wu\in P_1^{u}-P_1^w$, $us\notin P_1^{u}\cup P_1^s$, $sv\in P_1^s\cap P_4^v$ and $vw\in P_1^v\cap P_1^w$, we have that $(w,u,s,v,w)$ is an $H$-cycle in $G$ of length 4 containing $v$, which is impossible.
	\end{proof}
	
	By Claim~5, we have that $us\in P_1^{u}$ or $us\in P_1^{s}$.\\
	\hspace*{0.5cm}When $us\in P_1^{u}$, we have by Claim~3 that $\{vu,wu,xu,su\}\subseteq P_1^{u}$. Moreover, $\delta_G(u)=6$, so $k_u\le 3$, which is impossible.\\
	\hspace*{0.5cm}When $us\in P_1^{s}$, Claim~4 implies that $\{vs,ws,xs,us\}\subseteq P_1^{s}$. Since $\delta_G(s)=6$, it follows that $k_s\le 3$, a contradiction.
\end{description}
Therefore every vertex in $G$ is contained in an $H$-cycle of length 4, whenever $|V(G)|=7$.\\

\hspace*{0.5cm}\hspace*{0.5cm}The case $n=8$, can be easily carried as a consequence of the case $n=7$, since for each $x$ in $V(G)$, $G-x$ satisfies the hypotheses of Theorem~\ref{teo:4Hciclo-1}.\\

This concludes the proof of Theorem~\ref{teo:4Hciclo-1}. $\square$\\

\textbf{Proof of Theorem~\ref{teo:4Hciclo-2}}

Let $G$ and $H$ be graphs as in the hypotheses, thus for every $x$ en $V(G)$, $k_x\ge 5$. Proceeding by contradiction, suppose that there exists a vertex $v$ in $V(G)$ such that $v$ is not contained in an $H$-cycle of length 4. \\
It follows by Theorem~\ref{teo:3Hciclo} that there exists an $H$-cycle of length 3 containing $v$, say $T=(u,v,w,u)$. By Observation~\ref{obs:equiv}, we have that $uv$ and $vw$ are in different parts of the $k_v$-partition of $V(G_v)$, $vw$ and $wu$ are in different parts of the $k_w$-partition of $V(G_w)$, and $wu$ and $uv$ are in different part of the $k_u$-partition of $V(G_u)$. Without loss of generality, suppose that $uv\in P_1^v$ and $vw\in P_2^v$, $vw\in P_2^w$ and $wu\in P_1^w$, and $wu\in P_2^{u}$ and $uv\in P_1^{u}$.\\

The proof of Theorem~\ref{teo:4Hciclo-2} is divided into two cases, depending on whether there exists a vertex $x$ fulfilling certain properties or not. In each case, a series of claims is proven, which will lead us to a contradiction.\\

\noindent\textbf{Case 1.} There exists a vertex $x$ in $V(G)-\{u,v,w\}$ such that
\begin{itemize}
	\item $xu\in P_2^{u}$, $xw\in P_1^{w}$,
	\item $ux$ and $wx$ are in the same part of the $k_x$-partition of $V(G_x)$,
	\item $vx$ and $ux$ are in different part of the $k_x$-partition of $V(G_x)$,
	\item $xv\notin P_1^{v}\cup P_2^v$.
\end{itemize}

Suppose without loss of generality that $\{ux,wx\}\subseteq P_1^{x}$, $vx\in P_2^x$ and $xv\in P_3^v$.\\

\textbf{Claim 1.} For every $y$ in $V(G)-\{u,v,w,x\}$, if $yw\notin P_1^w\cup P_2^w$, then
\begin{description}
	\item(1) $uy$, $vy$, $wy$ and $xy$ are in the same part of the $k_y$-partition of $V(G_y)$,
	\item(2) $uy\notin P_2^{u}$, $xy\notin P_1^x$, and $vy\notin P_2^v$.
\end{description}
\begin{proof}
	Let $y$ be a vertex in $V(G)-\{u,v,w,x\}$ such that $yw\notin P_1^w\cup P_2^w$. Suppose that $wy\in P_1^y$, thus the item (1) can be rewritten as follows: $\{uy,vy,wy,xy\}\subseteq P_1^y$.\\ 
	
	First, proceeding by contradiction, we prove that $vy\in P_1^y$. \\
	\hspace*{0.5cm}If $vy\in P_1^v$, then by Observation~\ref{obs:equiv4} and given that $wx\in P_1^w\cap P_1^x$, $xv\in P_2^x\cap P_3^v$, $vy\in P_1^v-P_1^y$ and $yw\in P_1^y-P_1^w$, we have that $(w,x,v,y,w)$ is an $H$-cycle of length 4 in $G$ containing $v$. A contradiction.\\
	\hspace*{0.5cm}If $vy\notin P_1^v$, then by Observation~\ref{obs:equiv4} with $wu\in P_1^w\cap P_2^{u}$, $uv\in P_1^v\cap P_1^{u}$, $vy\notin P_1^v\cup P_1^y$ and $yw\in P_1^y-P_1^w$, we have that $(w,u,v,y,w)$ is an $H$-cycle of length 4 in $G$ containing $v$, which is impossible.\\
	We conclude that $vy\in P_1^y$.\\
	
	Again, proceeding by contradiction, we will prove that $vy\notin P_2^{v}$. \\
	It follows from Observation~\ref{obs:equiv} that $v$ is an obstruction of the cycle $C=(y,v,w,y)$, since $vw\in P_2^{v}$. By the same observation, we have that $y$ is an obstruction of the cycle $C$ (as $\{vy,wy\}\subseteq P_1^y$), and $w$ is not an obstruction of the cycle $C$ (recall that $vw\in P_2^w$ and $yw\notin P_2^w$). Thus, the cycle $C$ has exactly  2 obstructions (namely $v$ and $y$), with $\ell(C)=3$, which is impossible because of the hypothesis~3 of Theorem~\ref{teo:4Hciclo-2}. Therefore $vy\notin P_2^{v}$.\\
	
	Proceeding by contradiction, we will prove that $xy\in P_1^y$. \\
	Notice that, under our assumption $xy\in P_2^x$. Otherwise, if $xy\notin P_2^x$, then by Observation~\ref{obs:equiv4} and given that $wv\in P_2^v\cap P_2^w$, $vx\in P_2^x\cap P_3^v$, $xy\notin P_1^y\cup P_2^x$ and $yw\in P_1^y-P_2^w$, we have that $(w,v,x,y,w)$ is an $H$-cycle of length 4 in $G$ containing $v$, which is impossible. Hence, by Observation~\ref{obs:equiv4}, $(x,y,v,w,x)$ is an $H$-cycle of length 4 in $G$ containing $v$ (as $xy\in P_2^x-P_1^y$, $yv\in P_1^y- P_2^v$, $vw\in P_2^v\cap P_2^w$ and $wx\in P_1^w\cap P_1^x$), a contradiction. Therefore~$xy\in P_1^y$.\\
	
	Proceeding by contradiction, we will prove that $xy\notin P_1^{x}$. \\
	By Observation~\ref{obs:equiv}, we have the following assertions: $x$ is an obstruction of the cycle $C'=(y,x,w,y)$ (because $wx\in P_1^{x}$), $y$ is an obstruction of the cycle $C'$ (as $\{xy,wy\}\subseteq P_1^y$), and $w$ is not an obstruction of the cycle $C'$ (recall that $xw\in P_1^w$ and $yw\notin P_1^w$). Hence, the cycle $C'$ has exactly two obstructions (namely $x$ and $y$), with $\ell(C')=3$, contradicting the hypothesis~3 of Theorem~\ref{teo:4Hciclo-2}. Therefore $xy\notin P_1^{x}$.\\
	
	Proceeding by contradiction, we will see that $uy\in P_1^y$. \\
	Observe that, under our assumption, $uy\in P_1^u$. Otherwise, if $uy\notin P_1^{u}$, then by Observation~\ref{obs:equiv4}, we have that $(w,v,u,y,w)$ is an $H$-cycle of length 4 in $G$ containing $v$ (recall that $wv\in P_2^v\cap P_2^w$, $vu\in P_1^v\cap P_1^{u}$, $uy\notin P_1^{u}\cup P_1^y$ and $yw\in P_1^y-P_2^w$), which is impossible. Hence, by Observation~\ref{obs:equiv4}, $(u,y,v,w,u)$ is an $H$-cycle of length 4 in $G$ containing $v$ (since $uy\in P_1^{u}-P_1^y$, $yv\in P_1^y-P_2^v$, $vw\in P_2^v\cap P_2^w$ and $wu\in P_1^w\cap P_2^{u}$), a contradiction. Therefore $uy\in P_1^y$.\\
	
	Finally, proceeding by contradiction we will show that $uy\notin P_2^{u}$. \\
	Thus by Observation~\ref{obs:equiv}, we have the following three assertions: $u$ is an obstruction of the cycle $C''=(y,u,w,y)$ (because $uw\in P_2^{u}$), $y$ is an obstruction of the cycle $C''$ (as $\{uy,wy\}\subseteq P_1^y$), and $w$ is not an obstruction of the cycle $C''$ (recall that $uw\in P_1^w$ and $yw\notin P_1^w$). Hence, the cycle $C''$ has exactly two obstructions (namely $u$ and $y$), with $\ell(C)=3$, which is impossible because of the hypothesis~3 of Theorem~\ref{teo:4Hciclo-2}. Therefore $uy\notin P_2^{u}$.\\ 
	
	This concludes the proof of Claim~1.
\end{proof}

For every $i$ in $\{1,2,\cdots, k_w\}$, let $N_i^w=\{y\in N_G(w):yw\in P_i^w\}$ be, and consider $Z=\{y\in N_G(w)-(N_1^w\cup N_2^w):uy\in P_1^{u}\}$.

\textbf{Claim~2.} $Z\subseteq N_i^w$ for some $i$ in $\{3,\cdots, k_w\}$.

\begin{proof}
	Let $y_1$ be in $Z$, and w.l.o.g. assume that $uy_1\in P_i^{w}$ for some fixed $i$, $i>2$. We will prove that $Z\subseteq N_i^{w}$.\\
	Let $y_2$ be in $Z\setminus\{y_1\}$. Given that $\{y_1,y_2\}$ is a subset of $Z$, we have that $y_1\notin N_1^w\cup N_2^w$ and $y_2\notin N_1^w\cup N_2^w$, that is, $y_1 w\notin P_1^w\cup P_2^w$ and $y_2w\notin P_1^w\cup P_2^w$. By Claim~1, $wy_j$ and $uy_j$ are in the same part of the $k_{y_j}$-partition of $V(G_{y_j})$, for $j$ in $\{1,2\}$, so by Observation~\ref{obs:equiv}, $y_j$ is an obstruction of the path $(w,y_j,u)$. Also, as $\{y_1,y_2\}$ is a subset of $Z$, we have that  $\{uy_1,uy_2\}\subseteq P_1^{u}$, which implies by the previous observation that $u$ is an obstruction of the path $(y_1,u,y_2)$. Thus $u$, $y_1$ and $y_2$ are obstructions of the cycle $C=(w,y_1,u,y_2,w)$.\\
	If $w$ is not an obstruction of the cycle $C$, then the cycle $C$ has exactly 3 obstructions (namely $u$, $y_1$ y $y_2$), with $\ell(C)=4$, contradicting the hypothesis~2 of Theorem~\ref{teo:4Hciclo-2}, so $w$ is an obstruction of the cycle $C$. Hence, by Observation~\ref{obs:equiv}, $wy_1$ and $wy_2$ are in the same part of the $k_{w}$-partition of $V(G_w)$, and given that $wy_1$ is in $P_i^w$, it follows that $wy_2\in P_i^w$, which implies that $y_2\in N_i^w$.\\
	Therefore $Z\subseteq N_i^w$ for some $i$ in $\{3,\cdots, k_w\}$.
\end{proof}

Suppose without loss of generality that $Z\subseteq N_3^w$, and let $A=\bigcup_{i=4}^{k_w} N_i^w$ be and $S=V(G)-(A\cup\{u,v,w,x\})$. Notice that $|A|\neq 0$, since $k_w\geq 5$; moreover $|A|\geq 2$.

\textbf{Claim 3.} $A$ has the $H$-dependency property with respect to the vertex~$v$.

\begin{proof}
	Let $\{a,a'\}$ be a subset of $A$. As $\{a,a'\}$ is a subset of $A$, it follows that $a\notin N_1^w\cup N_2^w$ and $a'\notin N_1^w\cup N_2^w$, that is, $a w\notin P_1^w\cup P_2^w$ and $a' w\notin P_1^w\cup P_2^w$. Now, by Claim~1, $wa$ and $va$ are in the same part of the $k_a$-partition of $V(G_a)$, and $wa'$ and $va'$ are in the same part of the $k_{a'}$-partition of $V(G_{a'})$. Say that $\{va,wa\}\subseteq P_1^{a}$ and $\{va',wa'\}\subseteq P_1^{a'}$.\\
	Proceeding by contradiction, suppose that $A$ has not the $H$-dependency  property with respect to the vertex $v$, that is, $a$ is not an obstruction of the path $(v,a,a')$ and $a'$ is not an obstruction of the path $(v,a',a)$. This implies, by Observation~\ref{obs:equiv}, that $va$ and $aa'$ are in different parts of the $k_a$-partition of $V(G_a$), and $va'$ and $aa'$ are in different part of the $k_{a'}$-partition of $V(G_{a'})$. Since $va$ is in $P_1^{a}$ and $va'$ is in $P_1^{a'}$, we have that $aa'\notin P_1^{a}\cup P_1^{a'}$; moreover, as $wa'$ is not in $P_1^w\cup P_2^w$, it follows by Claim~1 that $va'\notin P_2^v$. Hence, by Observation~\ref{obs:equiv4} and given that $wa\in P_1^{a}-P_2^w$, $aa'\notin P_1^{a}\cup P_1^{a'}$, $a' v\in P_1^{a'}-P_2^v$ and $vw\in P_2^v\cap P_2^w$, it follows that $(w,a,a',v,w)$ is an $H$-cycle of length 4 in $G$ containing $v$, which is impossible by our assumption.\\
	Therefore $A$ has the $H$-dependency property with respect to the vertex~$v$.
\end{proof}

Given that $A$ has the $H$-dependency property with respect to the vertex $v$, we have by Proposition~\ref{prop:hdep} that there exists a vertex $a$ in $A$ such that $l_a^D\leq\frac{t+1}{2}$, where $D=G[A]$ and $t=|A|$, and $a$ is an obstruction of the path $(v,a,a')$ for some $a'$ in $N_D(a)$ (recall that $|A|\geq 2$). Since $a$ is in $A$, it follows that $a\notin N_1^w\cup N_2^w$, that is, $wa\notin P_1^w\cup P_2^w$, which implies by Claim~1 that $ua$, $va$, $wa$ and $xa$ are in the same part of the $k_a$-partition of $V(G_a)$, and $va\notin P_2^v$. Suppose that $va\in P_1^{a}$, so $\{ua,va,wa,xa\}\subseteq P_1^{a}$, moreover, as $a$ is an obstruction of the path $(v,a,a')$, it follows by Observation~\ref{obs:equiv} that $aa'\in P_1^{a}$.\\
Denoting by $s=min\{|S|,3\}$ (recall that $S=V(G)-(A\cup\{u,v,w,x\})$), $S_1=N_1^w-\{u,x\}$, $S_2=N_2^w-\{v\}$ and $S_3=N_3^w$, we have that $S=S_1\cup S_2\cup S_3$ and $n=|S|+|A|+|\{u,v,w,x\}|=|S|+t+4\ge s+t+4$.

\textbf{Claim~4.} $k_a\le l_a^D+s$ (see~Notation~\ref{not:lxd}).

\begin{proof}
	Recall that $G_a$ is a complete $k_a$-partite graph for some $k_a$ in $\mathbb{N}$, and $\mathcal{P}=\{P_1^{a},P_2^{a},\cdots, P_{k_a}^{a}\}$ is the $k_a$-partition of $V(G_a)$ into independent sets. By the Observation~\ref{obs:kxinducida}, since $D$ is an induced subgraph of $G$, we have that $\mathcal{Q}=\{P_i^{a}\cap V(D_a): P_i^{a}\cap V(D_a)\neq \emptyset, i\in \{1,2,\cdots, k_a\}\}$ is the $l_{a}^D$-partition of $V(D_a)$ into independent sets, where $V(D_a)=\{e\in E(D):e\ \mbox{is incident with}\ a\}$. Renaming, if necessary, we say that $\mathcal{Q}=\{P_1^{a}\cap V(D_a),P_2^{a}\cap V(D_a),\cdots, P_{l_a^D}^{a}\cap V(D_a)\}$ is the $l_a^D$-partition of $V(D_a)$ into independent sets.\\
	Now, in order to prove this upper bound of $k_a=k_a^{G}$ in terms of $l_a^{D}$, consider the edges incident with $a$ that are not vertices in $D_a$, that is, $V(G_a)-V(D_a)=\{ua,va,wa,xa\}\cup \{ay:y\in S\}$. Since $\{ua,va,wa,xa\}\subseteq P_1^{a}$ and $a'a\in P_1^{a}$, with $a'\in N_D(a)$, when we consider the edges $ua$, $va$, $wa$ and $xa$, we add no extra part in $\mathcal{P}$ with respect to the number of parts of $\mathcal{Q}$. Given this, it is enough to analyze what happens when we consider the edges $ya$, with $y\in S$. We will see that, when we consider such edges, we add at most three parts in the $k_a$-partition of $V(G_a)$ with respect to the number of parts in the $l_a^D$-partition of $V(D_a)$, one part for each set $S_i$, with $i$ in $\{1,2,3\}$.\\
	Let $y$ be in $S$, hence $y\in S_1\cup S_2\cup S_3$ (moreover, by the definition of the sets $S_1$, $S_2$ and $S_3$, we have that $yw\in P_1^w\cup P_2^w\cup P_3^w$). In view of this, we divide the rest of the proof of Claim~4 into 3 cases, depending on whether $y$ is in $S_1$, $S_2$ or $S_3$.
	
	\textbf{Case 4.1.} $y\in S_1$. \\
	Thus $wy\in P_1^w$, in particular $wy\notin P_2^w$. W.l.o.g. suppose that~$wy\in P_1^y$.
	
	\textbf{Claim 4.1.} For every $z$ in $S_1$, if $wz$ is in $P_1^z$, then $az\in P_1^z\cup P_1^{a}$.
	\begin{proof}
		Let $z$ be in $S_1$, thus $zw\in P_1^w$ (in particular $zw\notin P_2^w$) and w.l.o.g suppose that $wz\in P_1^z$. Proceeding by contradiction, suppose that $az\notin P_1^z\cup P_1^{a}$. Since $wv\in P_2^v\cap P_2^w$, $va\in P_1^{a}-P_2^v$, $az\notin P_1^z\cup P_1^{a}$ and $zw\in P_1^z- P_2^w$, we have by Observation~\ref{obs:equiv4}  that $(w,v,a,z,w)$ is an $H$-cycle of length 4 in $G$ containing $v$, which is impossible, and Claim~4.1 holds.
	\end{proof}
	
	Given that $y\in S_1$, with $wy\in P_1^y$, we have by Claim~4.1 that $ay\in P_1^y\cup P_1^{a}$. If $ay\in P_1^{a}$, then when we consider that edge, we add no extra part in the $k_a$-partition of $V(G_a)$ with respect to the number of parts of the $l_a^D$-partition of $V(D_a)$, as $a'a\in P_1^{a}$, with $a'\in N_D(a)$. Now, we can assume that $ay\notin P_1^{a}$, which implies that $ay\in P_1^y$ and $ay\in P_\alpha^{a}$ for some $\alpha\neq 1$ (moreover, we can assume that $\alpha>l_a^D$, otherwise, the edge $ay$ is already counted in the first $l_a^D$ parts of the partition of $V(G_a)$). 
	
	\textbf{Claim 4.2.} For every $y'$ in $S_1$, $ay'\in P_1^{a}$ or $ay'\in P_\alpha^{a}$.
	\begin{proof}
		Let $y'$ be in $S_1-\{y\}$, and w.l.o.g. suppose that $wy'\in P_1^{y'}$. By Claim~4.1, we have that $ay'\in P_1^{y'}\cup P_1^{a}$. Assuming that $ay'\notin P_1^{a}$, we will see that $ay'\in P_\alpha^{a}$. \\
		First notice that $ay'\in P_1^{y'}$. Since $\{ay',wy'\}$ is a subset of $P_1^{y'}$, it follows by Observation~\ref{obs:equiv} that $y'$ is an obstruction of the walk $(a,y',w)$. By the same argument, $y$ is an obstruction of the path $(a,y,w)$ (recall that $\{wy,ay\}\subseteq P_1^y$) and $w$ is an obstruction of the path $(y,w,y')$ (as $\{wy,wy'\}\subseteq P_1^w$). Thus $w$, $y$ and $y'$ are obstructions of the cycle $C=(w,y,a,y',w)$. If $a$ is an obstruction of the cycle $C$, then $C$ has exactly 3 obstructions (namely $w$, $y$ and $y'$), contradicting the hypothesis~2 of Theorem~\ref{teo:4Hciclo-2}, so $a$ is not an obstruction of the cycle $C$. By Observation~\ref{obs:equiv}, $ay$ and $ay'$ are in the same part of the $k_a$-partition of $V(G_a$), and since $ay$ is in $P_\alpha^{a}$, it follows that $ay'\in P_\alpha^{a}$, as desired.
	\end{proof}
	Since for every $y'$ in $S_1$, $ay'\in P_1^{a}$ or $ay'\in P_{\alpha}^{a}$, with $\alpha$ in $\{l_a^D+1,\cdots,k_{a}\}$, we have that, when we consider the edges $y'a$, with $y'$ in $S_1$, we add at most one part in $\mathcal{P}$ with respect to the number of parts in $\mathcal{Q}$, namely $P_{\alpha}^{a}$.
	
	\textbf{Case 4.2.} $y\in S_2$. \\
	Thus $wy\in P_2^w$, w.l.o.g. assume that $wy\in P_1^y$.
	
	\textbf{Claim 4.3.} For every $z$ in $S_2$, if $wz$ is in $P_1^z$, then $vz\in P_1^z$ and $az\in P_1^z\cup P_1^{a}$.
	\begin{proof}
		Let $z$ be in $S_2$, thus $wz\in P_2^w$ and w.l.o.g. suppose that $wz\in P_1^z$. \\
		Now, proceeding by contradiction we will prove that $vz\in P_1^z$; so suppose that $vz\notin P_1^z$.\\
		\hspace*{0.5cm}If $vz\in P_3^v$, then by Observation~\ref{obs:equiv} and given that $wu\in P_1^w\cap P_2^{u}$, $uv\in P_1^v\cap P_1^{u}$, $vz\in P_3^v-P_1^z$ and $zw\in P_1^z\cap P_2^w$, we have that $(w,u,v,z,w)$ is an $H$-cycle of length 4 in $G$ containing $v$, which is impossible.\\
		\hspace*{0.5cm}If $vz\notin P_3^v$, then by the same observation with $wx\in P_1^w\cap P_1^x$, $xv\in P_2^x\cap P_3^v$, $vz\notin P_3^v\cup P_1^z$ and $zw\in P_1^z\cap P_2^w$, we have that $(w,x,v,z,w)$ is an $H$-cycle of length 4 in $G$ containing $v$, a contradiction.\\
		We conclude that $vz\in P_1^z$.\\
		
		Now, we will see that $az\in P_1^z\cup P_1^{a}$. Proceeding by contradiction, suppose that $az\notin P_1^z\cup P_1^{a}$.\\
		\hspace*{0.5cm}If $vz\in P_2^v$, then by Observation~\ref{obs:equiv} and given that $vu\in P_1^v$, it follows that $v$ is not an obstruction of the path $(z,v,u)$. By the same argument, $a$ is not an obstruction of the path $(u,a,z)$ (as $ua\in P_1^{a}$ and $az\notin P_1^{a}$), and $z$ is not an obstruction of the path $(a,z,v)$ (recall that $zv\in P_1^z$ and $az\notin P_1^z$). Since $a$ is in $A=\bigcup_{i=4}^{k_w}N_i^w$, it follows that $a\notin N_3^w$, with $Z\subseteq N_3^w$ (see Claim~2), implying that $a\notin Z$. Hence, by the definition of the set $Z$, we have that $ua\notin P_1^{u}$, and given that $vu$ is in $P_1^{u}$, we have by the Observation~\ref{obs:equiv} that $u$ is not an obstruction of the path $(v,u,a)$. Thus $z$, $v$, $u$ and $a$ are not obstructions of the cycle $C=(z,v,u,a,z)$, and by Observation~\ref{obs:equiv4}, $C$ is an $H$-cycle of length 4 in $G$ containing $v$, which is impossible.\\
		\hspace*{0.5cm}If $vz\notin P_2^v$, then by Observation~\ref{obs:equiv4} we have that $(w,v,z,a,w)$ is an $H$-cycle of length 4 in $G$ containing $v$ (since $wv\in P_2^v\cap P_2^w$, $vz\in P_1^z-P_2^v$, $za\notin P_1^z\cup P_1^{a}$ and $aw\in P_1^{a}-P_2^w$), a contradiction. \\
		We conclude that $az\in P_1^z\cup P_1^{a}$.
	\end{proof} 
	
	Since $y\in S_2$, with $wy\in P_1^y$, we have by Claim~4.3 that $ay\in P_1^y\cup P_1^{a}$. If $ay\in P_1^{a}$, then when we consider this edge, we add no extra part in the $k_a$-partition of $V(G_a)$ with respect to the number of parts of the $l_a^D$-partition of $V(D_a)$, as $a'a\in P_1^{a}$, with $a'\in N_D(a)$. Hence, we can assume that $ay\notin P_1^{a}$, which implies that $ay\in P_1^y$ and $ay\in P_\alpha^{a}$ for some $\alpha\neq 1$ (moreover, we can assume that $\alpha>l_a^D$, otherwise, the edge $ay$ is already counted in the first $l_a^D$ parts of the partition of V($G_a$)).\\
	Proceeding in a very similar way to the proof of Claim~4.2 (taking $S_2$ instead of $S_1$ and $P_2^w$ instead of $P_1^w$) we can prove that:
	
	\textbf{Claim 4.4.} For every $y'$ in $S_2$, $ay'\in P_1^{a}$ or $ay'\in P_\alpha^{a}$.\\
	
	Since for every $y'$ in $S_2$, $ay'\in P_1^{a}$ or $ay'\in P_{\alpha}^{a}$, with $\alpha$ in $\{l_a^D +1,\cdots,k_a\}$, we have that, when we consider the edges $y'a$ with $y'$ in $S_2$, we add at most one part in $\mathcal{P}$ with respect to the number of parts in $\mathcal{Q}$, namely $P_{\alpha}^{a}$.
	
	\textbf{Case 4.3.} $y\in S_3$.\\
	Thus $wy\in P_3^w$, in particular $wy\notin P_2^w$. To prove the Case~4.1 of Claim~4, we only need that $wy\notin P_2^w$ as an extra information. Hence, the proof of this Case~4.3 is exactly the same to the proof of the Case~4.1.\\
	Therefore, when we consider the edges $ya$, with $y\in S_3$, we add at most one part in $\mathcal{P}$ with respect to the number of parts in $\mathcal{Q}$.\\
	
	This concludes the proof of Claim~4.
\end{proof}

Now we finish the proof of the Case~1 of Theorem~\ref{teo:4Hciclo-2}. \\
By Claim~4, we have that $l_a^D+s\ge k_a$, and since $k_a\ge \frac{n+1}{2}$ (by hypothesis) and $\frac{t+1}{2}\ge l_a^D$, it follows that $$\frac{t+1}{2}+s\ge l_a^D+s\ge k_a\ge \frac{n+1}{2}.$$ Simplifying this inequality we obtain that $t+2s\ge n$, and since $n\ge s+t+4$, we have that $t+2s\ge s+t+4$, that is, $s\ge 4$, which is impossible because $s=min\{|S|,3\}\le 3$.\\

This concludes the Case 1.\\

\noindent\textbf{Case 2.} There exists no vertex $x$ in $V(G)-\{u,v,w\}$ with the following properties
\begin{itemize}
	\item $xu\in P_2^{u}$, $xw\in P_1^{w}$,
	\item $ux$ and $wx$ are in the same part of the $k_x$-partition of $V(G_x)$,
	\item $vx$ and $ux$ are in different part of the $k_x$-partition of $V(G_x)$,
	\item $xv\notin P_1^{v}\cup P_2^v$.
\end{itemize}

\textbf{Claim 5.} For every $x$ in $V(G)-\{u,v,w\}$, if $vx\notin P_1^v\cup P_2^v$, then
\begin{enumerate}
	\item $ux$, $vx$ y $wx$ are in the same part of the $k_x$-partition of $V(G_x)$,
	\item $wx\notin P_2^w$ and $ux\notin P_1^{u}$.
\end{enumerate}
\begin{proof}
	Let $x$ be a vertex in $V(G)-\{u,v,w\}$ such that $vx\notin P_1^v\cup P_2^v$. Suppose without loss of generality that $vx\in P_1^x\cap P_3^v$, thus the item (1) can be rewritten as follows: $\{ux,vx,wx\}\subseteq P_1^x$.\\
	
	To prove Claim~5, first we establish the two following claims. \\
	
	\textbf{Claim~5.1.} $wx\in P_1^w\cup P_1^x$ and $ux\in P_2^{u}\cup P_1^x$.
	\begin{proof}
		First, proceeding by contradiction we will prove that $wx\in P_1^w\cup P_1^x$. So suppose that $wx\notin P_1^w\cup P_1^x$. Since $wu\in P_1^w\cap P_2^{u}$, $uv\in P_1^{u}\cap P_1^v$, $vx\in P_3^v\cap P_1^x$ and $xw\notin P_1^w\cup P_1^x$, we have by Observation~\ref{obs:equiv4} that $(w,u,v,x,w)$ is an $H$-cycle of length 4 in $G$ containing $v$, which is impossible. Therefore $wx\in P_1^w\cup P_1^x$.\\
		
		Now, we will see that $ux\in P_2^{u}\cup P_1^x$. Proceeding by contradiction, suppose that $ux\notin P_2^{u}\cup P_1^x$. Given that $wu\in P_1^w\cap P_2^{u}$, $ux\notin P_2^{u}\cup P_1^x$, $xv\in P_3^v\cap P_1^x$ and $vw\in P_2^v\cap P_2^w$, it follows by Observation~\ref{obs:equiv4} that $(w,u,x,v,w)$ is an $H$-cycle of length 4 in $G$ containing $v$, a contradiction. Therefore $ux\in P_2^{u}\cup P_1^x$.
	\end{proof}
	
	\textbf{Claim~5.2.} If $wx\in P_1^x$, then $wx\notin P_2^w$, and if $ux\in P_1^x$, then $ux\notin P_1^{u}$.
	\begin{proof}
		First, we will prove that $wx\in P_1^x$ implies that $wx\notin P_2^w$. Proceeding by contradiction, suppose that $wx\in P_1^x\cap P_2^w$. Since $\{wx,vx\}$ is a subset of $P_1^x$, we have by Observation~\ref{obs:equiv} that $x$ is an obstruction of the path $(v,x,w)$. By the same observation, $w$ is an obstruction of the path $(v,w,x)$ (since $\{wx,wv\}\subseteq P_2^w$), and $v$ is not an obstruction of the path $(w,v,x)$ (recall that $vw\in P_2^v$ and $vx\in P_3^v$). Hence, $(v,x,w,v)$ is a cycle of length 3 in $G$ with exactly two obstructions (namely $w$ and $x$), which is impossible because of the hypothesis~3 of Theorem~\ref{teo:4Hciclo-2}. Therefore $wx\notin P_2^w$, whenever $wx\in P_1^x$.\\
		
		Now, we will see that $ux\in P_1^x$ implies that $ux\notin P_1^{u}$. Proceeding by contradiction, suppose that $ux\in P_1^x\cap P_1^{u}$. Given that $\{ux,vx\}$ is a subset of $P_1^x$, it follows by Observation~\ref{obs:equiv} that $x$ is an obstruction of the path $(v,x,u)$. By the same observation, $u$ is an obstruction of the path $(v,u,x)$ (as $\{ux,uv\}\subseteq P_1^{u}$), and $v$ is not an obstruction of the path $(u,v,x)$ (since $vu\in P_1^v$ and $vx\in P_3^v$). Hence, $(v,x,u,v)$ is a cycle of length 3 in $G$ with exactly two obstructions (namely $u$ and $x$), contradicting the hypothesis~3 of Theorem~\ref{teo:4Hciclo-2}. Therefore $ux\in P_1^x$ implies $ux\notin P_1^{u}$.
	\end{proof}
	
	To conclude, we split the proof of Claim~5 into 4 cases with respect to the set $P_1^x$.\\
	
	\textbf{Case 5.1.} $\{wx,ux\}\subseteq P_1^x$.\\
	Since $\{wx,ux\}\subseteq P_1^x$, Claim~5.2 implies that $wx\notin P_2^w$ and $ux\notin P_1^{u}$. Also, by the assumption of this case, $\{ux,vx,wx\}\subseteq P_1^x$. In this case, Claim~5 holds.\\
	
	\textbf{Case 5.2.} $wx\in P_1^x$ and $ux\notin P_1^x$.\\
	Given that $ux\in P_2^{u}\cup P_1^x$ (see Claim~5.1) and $ux\notin P_1^x$, we have that $ux\in P_2^{u}$; also, as $wx$ is in $P_1^x$, we have by Claim~5.2 that $wx\notin P_2^w$. Hence, by Observation~\ref{obs:equiv4}, $(v,u,x,w,v)$ is an $H$-cycle of length 4 in $G$ containing $v$ (since $vu\in P_1^v\cap P_1^{u}$, $ux\in P_2^{u}-P_1^{x}$, $xw\in P_1^{x}-P_2^{w}$ and $wv\in P_2^v\cap P_2^w$), which is a contradiction. Therefore this case is not possible.\\
	
	\textbf{Case 5.3.} $wx\notin P_1^x$ and $ux\in P_1^x$.\\	
	By the Claim~5.1, we have that $wx\in P_1^{w}\cup P_1^x$, but $wx\notin P_1^x$, thus $wx\in P_1^{w}$; also, as $ux$ is in $P_1^x$, it follows by Claim~5.2 that $ux\notin P_1^u$. Hence, by Observation~\ref{obs:equiv4}, $(v,u,x,w,v)$ is an $H$-cycle of length 4 in $G$ containing $v$ (recall that $vu\in P_1^v\cap P_1^{u}$, $ux\in P_1^{x}-P_1^{u}$, $xw\in P_1^{w}-P_1^{x}$ and $wv\in P_2^v\cap P_2^w$), a contradiction. Therefore this case is impossible.\\
	
	\textbf{Case 5.4.} $wx\notin P_1^x$ and $ux\notin P_1^x$.\\
	By Claim~5.1 we know that $wx\in P_1^w\cup P_1^x$ and $ux\in P_2^{u}\cup P_1^x$, but $wx\notin P_1^x$ and $ux\notin P_1^x$, thus $wx\in P_1^w$ and $ux\in P_2^{u}$. Since $vu\in P_1^{u}$ and $xu\in P_2^{u}$, we have by Observation~\ref{obs:equiv} that $u$ is not an obstruction of the path $(v,u,x)$. By the same observation, $v$ is not an obstruction of the path $(w,v,u)$ (given that $uv\in P_1^v$ and $wv\in P_2^v$), and $w$ is not an obstruction of the path $(v,w,x)$ (as $xw\in P_1^w$ and $vw\in P_2^w$). So $u$, $v$ and $w$ are not obstructions of the cycle $C=(v,w,x,u,v)$. If $x$ is not an obstruction of the cycle $C$, then by the Observation~\ref{obs:equiv4}, $C$ is an $H$-cycle of length 4 in $G$ containing $v$, which is impossible. Hence, $v$ is not an obstruction of the cycle $C$, which implies by Observation~\ref{obs:equiv} that $ux$ and $wx$ are in the same part of the $k_x$-partition of $V(G_x)$. Recall that $xu\in P_2^{u}$, $xw\in P_1^w$, and $xv\notin P_1^v\cup P_2^v$; also $ux\notin P_1^x$ and $vx\in P_1^x$ (that is, $ux$ and $vx$ are in different part of the $k_x$-partition of $V(G_x)$). This contradicts the assumption of the Case~2 of Theorem~\ref{teo:4Hciclo-2}. Therefore this case is not possible.\\
	
	This concludes the proof of Claim~5.
\end{proof} 

For every $i$ in $\{1,2,\cdots, k_v\}$, let $N_i^v=\{y\in N(v):zv\in P_i^v\}$ be, and consider $A=\bigcup_{i=3}^{k_v} N_i^v$ and $S=V(G)-(A\cup\{u,v,w\})$. Notice that $|A|\neq 0$, since $k_v\geq 5$; moreover $|A|\geq 2$.\\

\textbf{Claim 6.} $A$ has the $H$-dependency property with respect to the vertex~$v$.

\begin{proof}
	Let $\{a,a'\}$ be a subset of $A$. As $\{a,a'\}$ is a subset of $A$, it follows that $a\notin N_1^v\cup N_2^v$ and $a'\notin N_1^v\cup N_2^v$, that is, $a v\notin P_1^v\cup P_2^v$ and $a' v\notin P_1^v\cup P_2^v$. Now, by Claim~5, $va$ and $wa$ are in the same part of the $k_a$-partition of $V(G_a)$, and $va'$ and $wa'$ are in the same part of the $k_{a'}$-partition of $V(G_{a'})$. Suppose without loss of generality that $\{va,wa\}\subseteq P_1^{a}$ and $\{va',wa'\}\subseteq P_1^{a'}$.\\
	Proceeding by contradiction, suppose that $A$ has not the $H$-dependency property with respect to the vertex $v$, that is, $a$ is not an obstruction of the path $(v,a,a')$ and $a'$ is not an obstruction of the path $(v,a',a)$. This implies by Observation~\ref{obs:equiv}, that $va$ and $aa'$ are in different parts of the $k_a$-partition of $V(G_a$), and $va'$ and $aa'$ are in different part of the $k_{a'}$-partition of $V(G_{a'})$. Since $va$ is in $P_1^{a}$ and $va'$ is in $P_1^{a'}$, we have that $aa'\notin P_1^{a}\cup P_1^{a'}$; moreover, as $va$ is not in $P_1^v\cup P_2^v$, it follows by Claim~5 that $wa\notin P_2^w$. Next, by Observation~\ref{obs:equiv4} and given that $wa\in P_1^{a}-P_2^w$, $aa'\notin P_1^{a}\cup P_1^{a'}$, $a' v\in P_1^{a'}-P_2^v$ and $vw\in P_2^v\cap P_2^w$, we have that $(w,a,a',v,w)$ is an $H$-cycle of length 4 in $G$ containing $v$, which is impossible.\\
	Therefore $A$ has the $H$-dependency property with respect to the vertex~$v$.\\
\end{proof}

Given that $A$ has the $H$-dependency property with respect to the vertex $v$, we have by Proposition~\ref{prop:hdep} that there exists a vertex $a$ in $A$ such that $l_a^D\leq\frac{t+1}{2}$, where $D=G[A]$ and $t=|A|$, and $a$ is an obstruction of the path $(v,a,a')$ for some $a'$ in $N_D(a)$ (recall that $|A|\geq 2$). Since $a$ is in $A$, thus $a\notin N_1^v\cup N_2^v$, that is, $va\notin P_1^v\cup P_2^v$, this implies by Claim~5 that $ua$, $va$ and $wa$ are in the same part of the $k_a$-partition of $V(G_a)$, $wa\notin P_2^w$ and $ua\notin P_1^{u}$. Suppose that $va$ is in $P_1^{a}$, so $\{ua,va,wa\}\subseteq P_1^{a}$, moreover, as $a$ is an obstruction of the path $(v,a,a')$, it follows by Observation~\ref{obs:equiv} that $aa'\in P_1^{a}$.\\
Denoting by $s=min\{|S|,2\}$, $S_1=N_1^v-\{u\}$ and $S_2=N_2^v-\{w\}$, it follows that $S=S_1\cup S_2$ and $n=|S|+|A|+|\{u,v,w\}|=|S|+t+3\ge s+t+3$.\\

\textbf{Claim 7.} $k_a\le l_a^D+s$ (see~Notation~\ref{not:lxd}).

\begin{proof}
	Recall that $G_a$ is a complete $k_a$-partite graph for some $k_a$ in $\mathbb{N}$, and $\mathcal{P}=\{P_1^{a},P_2^{a},\cdots, P_{k_a}^{a}\}$ is the $k_a$-partition of $V(G_a)$ into independent sets. As $D$ is an induced subgraph of $G$, by Observation~\ref{obs:kxinducida} we have that $\mathcal{Q}=\{P_i^{a}\cap V(D_a): P_i^{a}\cap V(D_a)\neq \emptyset, i\in \{1,2,\cdots, k_a\}\}$ is the $l_{a}^D$-partition of $V(D_a)$ into independent sets, where $V(D_a)=\{e\in E(D):e\ \mbox{incide en}\ a\}$. Renaming, if necessary, we say that $\mathcal{Q}=\{P_1^{a}\cap V(D_a),P_2^{a}\cap V(D_a),\cdots, P_{l_a^D}^{a}\cap V(D_a)\}$ is the $l_a^D$-partition of $V(D_a)$ into independent sets.\\
	Now, in order to prove this upper bound of $k_a=k_a^{G}$ in terms of $l_a^{D}$, consider the edges incident with $a$ that are not vertices in $D_a$, that is, $V(G_a)-V(D_a)=\{ua,va,wa\}\cup \{ya:y\in S\}$. Since $\{ua,va,wa\}\subseteq P_1^{a}$ and $a'a\in P_1^{a}$, with $a'\in N_D(a)$, when we consider the edges $ua$, $va$ and $wa$, we add no extra part in $\mathcal{P}$ with respect to the number of parts of $\mathcal{Q}$. In view of this, it suffices to analyze what happens when we consider the edges $ya$, with $y\in S$. We will see that, when we consider such edges, we add at most 2 parts in the $k_a$-partition of $V(G_a)$ with respect to the number of parts of the $l_a^D$-partition of $V(D_a)$, one part for each set $S_i$, with $i$ en $\{1,2\}$.\\
	Let $y$ in $S$, hence $y\in S_1\cup S_2$ (moreover, by the definition of the sets $S_1$ and $S_2$, we have that $yv\in P_1^v\cup P_2^v$). We divide the rest of the proof of Claim 7 into two cases, depending on whether $y$ is in $S_1$ or $S_2$.\\
	
	\textbf{Case 7.1.} $y\in S_1$. \\
	Thus $vy\in P_1^v$ and suppose without loss of generality that $wy\in P_1^y$.\\
	
	\textbf{Claim~7.1.} For every $z$ in $S_1$, if $vz$ is in $P_1^z$, then $az\in P_1^z\cup P_1^{a}$.
	\begin{proof}
		Let $z$ be in $S_1$, thus $vz\in P_1^v$ and suppose w.l.o.g. that $vz\in P_1^z$. Proceeding by contradiction, suppose that $az\notin P_1^z\cup P_1^{a}$. Since $vz\in P_1^v\cap P_1^z$, $za\notin P_1^{z}\cup P_1^{a}$, $aw\in P_1^{a}-P_2^w$ and $wv\in P_2^v\cap P_2^w$, we have by Observation~\ref{obs:equiv4} that $(v,z,a,w,v)$ is an $H$-cycle of length 4 in $G$ containing $v$, which is impossible and Claim~7.1 holds.
	\end{proof}
	
	Given that $y\in S_1$, with $vy\in P_1^y$, we have by Claim~7.1 that $ay\in P_1^y\cup P_1^{a}$. When $ay\in P_1^{a}$, and we consider such edge, no extra part appears in the $k_a$-partition of $V(G_a)$, with respect to the number of parts of the $l_a^D$-partition of $V(D_a)$, as $a'a\in P_1^{a}$, with $a'\in N_D(a)$. Now, we can assume that $ay\notin P_1^{a}$, which implies that $ay\in P_1^y$ and $ay\in P_\alpha^{a}$ for some $\alpha\neq 1$ (moreover, we can assume that $\alpha>l_a^D$, otherwise, the edge $ay$ is already counted in the first $l_a^D$ parts of $V(G_a)$). 
	
	\textbf{Claim 7.2.} For every $y'$ in $S_1$, $ay'\in P_1^{a}$ or $ay'\in P_\alpha^{a}$.
	\begin{proof}
		Let $y'$ be in $S_1-\{y\}$, thus $vy'\in P_1^v$, and (w.l.o.g.) suppose that $vy'\in P_1^{y'}$. By Claim~7.1 we have that $ay'\in P_1^{y'}\cup P_1^{a}$. If $ay'\in P_1^{a}$ we are done, so suppose that $ay'\notin P_1^{a}$, which implies that $ay'\in P_1^{y'}$. We will see that $ay'\in P_\alpha^{a}$. \\
		Since $\{ay',vy'\}$ is a subset of $P_1^{y'}$, it follows by Observation~\ref{obs:equiv} that $y'$ is an obstruction of the path $(a,y',v)$ in $G$. By the same argument, $y$ is an obstruction of the path $(a,y,v)$ (as $\{vy,ay\}\subseteq P_1^y$) and $v$ is an obstruction of the path $(y,v,y')$ (recall that $\{vy,vy'\}\subseteq P_1^v$). So $v$, $y$ and $y'$ are obstructions of the cycle $C=(v,y,a,y',v)$. If $a$ is not an obstruction of the cycle $C$, then the cycle $C$ has exactly 3 obstructions (namely $v$, $y$ and $y'$), contradicting the hypothesis~2 of Theorem~\ref{teo:4Hciclo-2}, so $a$ is an obstruction of the cycle $C$. By Observation~\ref{obs:equiv}, $ay$ and $ay'$ are in the same part of the $k_a$-partition of $V(G_a$), and since $ay$ is in $P_\alpha^{a}$, it follows that $ay'\in P_\alpha^{a}$, as desired.
	\end{proof}
	Since for every $y'$ in $S_1$, $ay'\in P_1^{a}$ or $ay'\in P_\alpha^{a}$, with $\alpha$ in $\{l_a^D +1,\cdots, k_a\}$, we have that, when we consider the edges $ay'$, with $y'$ in $S_1$, we add at most one part in $\mathcal{P}$ with respect to the number of parts in $\mathcal{Q}$, namely $P_{\alpha}^{a}$.
	
	\textbf{Case 7.2.} $y\in S_2$. \\
	Thus $vy\in P_2^v$ and suppose without loss of generality that $wy\in P_1^y$.\\
	
	\textbf{Claim 7.3.} For every $z$ in $S_2$, if $vz$ is in $P_1^z$, then $az\in P_1^z\cup P_1^{a}$.
	\begin{proof}
		Let $z$ be in $S_2$, thus $vz\in P_2^v$ and suppose (w.l.o.g.) that $vz\in P_1^z$. Proceeding by contradiction, suppose that $az\notin P_1^z\cup P_1^{a}$. Since $vz\in P_2^v\cap P_1^z$, $za\notin P_1^{z}\cup P_1^{a}$, $au\in P_1^{a}-P_1^u$ and $uv\in P_1^v\cap P_1^{u}$, we have by Observation~\ref{obs:equiv4} that $(v,z,a,u,v)$ is an $H$-cycle of length 4 in $G$ containing $v$, a contradiction and Claim~7.3 holds.
	\end{proof}
	
	Given that $y\in S_2$, with $vy\in P_1^y$, we have by Claim~7.3 that $ay\in P_1^y\cup P_1^{a}$. If $ay\in P_1^{a}$, then when we consider that edge, we add no extra part in the $k_a$-partition of $V(G_a)$ with respect to the number of parts of the $l_a^D$-partition of $V(D_a)$, as $a'a\in P_1^{a}$, with $a'\in N_D(a)$. Now, we can assume that $ay\notin P_1^{a}$, which implies that $ay\in P_1^y$ and $ay\in P_\alpha^{a}$ for some $\alpha\neq 1$ (moreover, we can assume that $\alpha>l_a^D$, otherwise, the edge $ay$ is already counted in the first $l_a^D$ parts of $G_a$). \\
	Proceeding in a very similar way to the proof of Claim~7.2 (taking $S_2$ instead of $S_1$, and $P_2^v$ instead of $P_1^v$), we can prove the following assertion:
	
	\textbf{Claim 7.4.} For every $y'$ in $S_2$, $ay'\in P_1^{a}$ or $ay'\in P_\alpha^{a}$.\\
	
	Since for every $y'$ in $S_2$, $ay'\in P_1^{a}$ or $ay'\in P_\alpha^{a}$, with $\alpha$ in $\{l_a^D +1,\cdots, k_a\}$, we have that, when we consider the edges $ay'$, with $y'$ in $S_2$, we add at most one part in $\mathcal{P}$ with respect to the number of parts in $\mathcal{Q}$, namely $P_{\alpha}^{a}$.\\
	
	This concludes the proof of Claim~7.
\end{proof}

Now we finish the proof of the Case~2 of Theorem~\ref{teo:4Hciclo-2}. \\
By Claim~7, we have that $l_a^D+s\ge k_a$, and as $k_a\ge \frac{n+1}{2}$ (by hypothesis) and $\frac{t+1}{2}\ge l_a^D$, it follows that $$\frac{t+1}{2}+s\ge l_a^D+s\ge k_a\ge \frac{n+1}{2}.$$ Simplifying this inequality we obtain that $t+2s\ge n$, and since $n\ge s+t+3$, we have that $t+2s\ge s+t+3$, that is, $s\ge 3$, which is impossible because $s=min\{|S|,2\}\le 2$.\\

This concludes the Case 2.\\

Therefore, as we assumed that the conclusion of Theorem~\ref{teo:4Hciclo-2} is false, and each case led us to a contradiction, we conclude that each vertex of $G$ is contained in an $H$-cycle of length 4 in $G$. $\square$\\

We show some consequences of the main results of this paper, which are well-known results in the theory of properly colored walks. We need the following definition and the Observation~\ref{obs:general}.

\begin{defi}\label{defi:colordegree}
	Let $G$ be an edge-colored graph and $x$ be a vertex in $G$. We define the color degree of the vertex $x$, denoted by $\delta^c(x)$, as follows: $\delta^c(x)=|\{c(e):e\ \mbox{is an edge incident with}\ x\}|$.
\end{defi}

\begin{obs}\label{obs:general}
	Let $H$ be the complete graph without loops and $G$ be an $H$-colored graph. For every $u$ in $V(G)$:
	\begin{enumerate}
		\item $G_u$ is a complete $k_u$-partite graph for some $k_u$ in $\mathbb{N}$, moreover, $k_u=\delta^c(u)$.
		\item If $W$ is a walk in $G$, then $W$ is a properly colored walk if and only if $W$ is an $H$-walk. In particular, $W$ is a properly colored path (cycle) if and only if $W$ is an $H$-path ($H$-cycle).
		\item The graph $G$ does not contain a cycle of length 3 with exactly 2 obstructions.
		\item The graph $G$ does not contain a cycle of length 4 with exactly 3 obstructions.
	\end{enumerate}
\end{obs}

\begin{cor}[\cite{SFC}]\label{cor:1}
	Let $G$ be a edge-colored complete graph of order $n$. If $n\ge 3$ and $\delta^c(x)\ge \frac{n+1}{2}$ for every $x$ in $V(G)$, then each vertex is contained in a properly colored cycle of length 3.
\end{cor}

\begin{cor}[\cite{SFC}]\label{cor:2}
	Let $G$ be a edge-colored complete graph of order $n$. If $n\ge 4$ and $\delta^c(x)\ge \frac{n+1}{2}$ for every $x$ in $V(G)$, then each vertex is contained in a properly colored cycle of length 4.
\end{cor}

\begin{remark}
	The edge-colored complete graph $G$ of the Figure~\ref{fig:contraejemplo} shows that, Corollary~\ref{cor:1} does not imply Theorem~\ref{teo:3Hciclo}, and Corollary~\ref{cor:2} does not imply Corollary~\ref{cor:4Hciclo}. 
\end{remark}

\begin{figure}[h!]
	\begin{center}
		\includegraphics[scale=0.4]{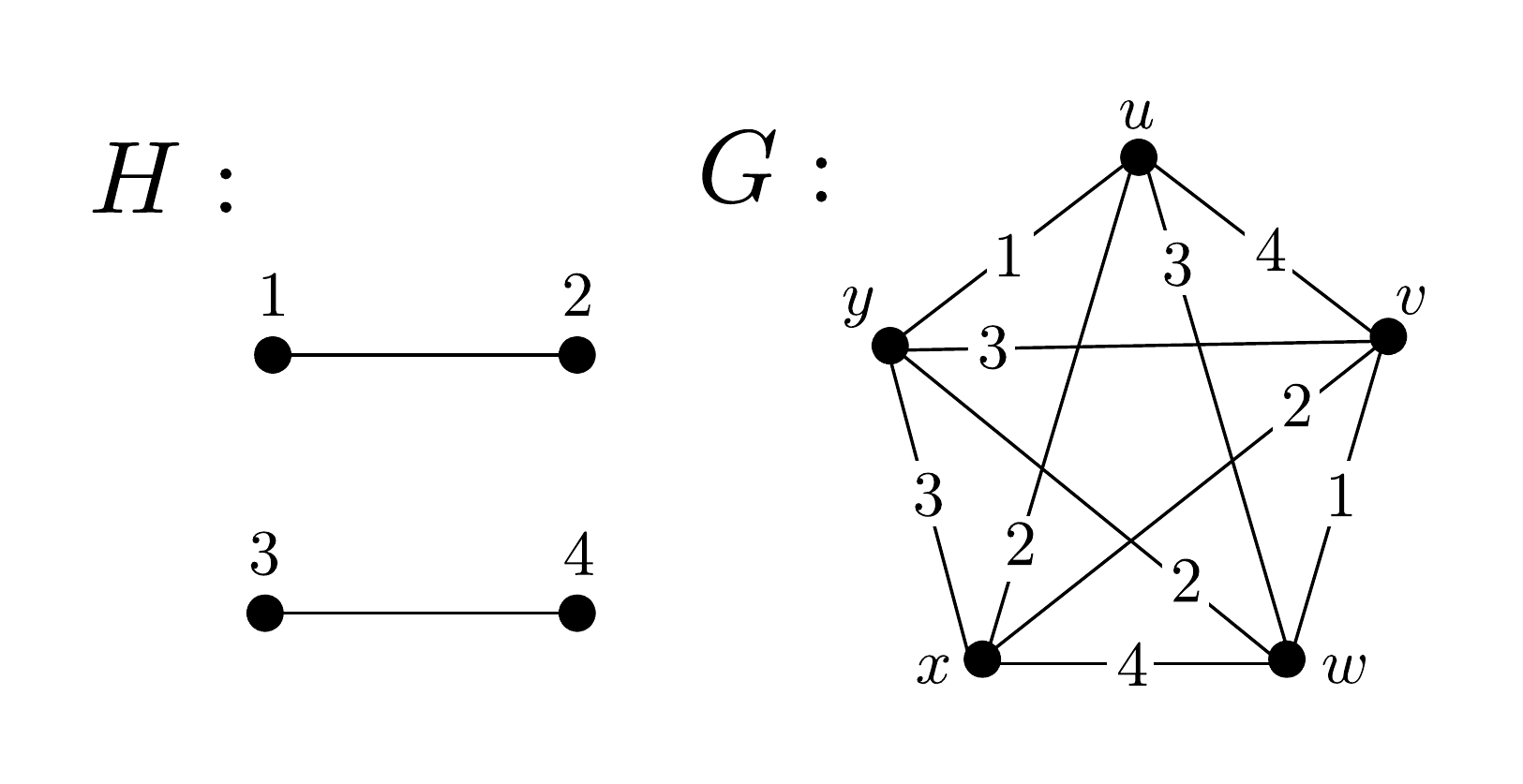}
		\caption{Corollary~\ref{cor:1} does not imply Theorem~\ref{teo:3Hciclo}, and Corollary~\ref{cor:2} does not imply Corollary~\ref{cor:4Hciclo}.}
		\label{fig:contraejemplo}
	\end{center}
\end{figure}


\end{document}